\documentclass[onefignum]{siamart190516}

\usepackage{lipsum}
\usepackage{amsfonts}
\usepackage{graphicx}
\usepackage{epstopdf}
\usepackage{algorithmic}
\ifpdf
  \DeclareGraphicsExtensions{.eps,.pdf,.png,.jpg}
\else
  \DeclareGraphicsExtensions{.eps}
\fi

\newcommand{\pp}{\partial_+}
\newcommand{\pn}{\partial_-}
\def\tilde{\widetilde}

\def\TS{\textstyle}

\def\nn{\nonumber}
\def\ol{\overline}

\def\dbyd#1{\TS{\frac{\partial}{\partial#1}}}

\newcommand{\ga}{\gamma}

\newcommand{\RR}{\mathbb{R}}
\newcommand{\beq}{\begin{equation}}
\newcommand{\eeq}{\end{equation}}
\newcommand{\eps}{\varepsilon}
\newcommand{\dd}{{\rm d}}

\newsiamremark{remark}{Remark}
\newsiamremark{hypothesis}{Hypothesis}
\crefname{hypothesis}{Hypothesis}{Hypotheses}
\newsiamthm{claim}{Claim}
\newsiamthm{condition}{Condition}
\newsiamthm{example}{Example}
\newsiamthm{assumption}{Assumption}
\numberwithin{equation}{section}
\numberwithin{figure}{section}

\makeatletter \@addtoreset{equation}{section} \makeatother
\renewcommand{\theequation}{\thesection.\arabic{equation}}

\headers{Singularities and Global Continuous Solutions}{G. Chen, G.-Q. Chen, and S. Zhu}

\title{Formation of Singularities
and Existence of \\ Global Continuous Solutions \\
for the Compressible Euler Equations
\thanks{Submitted to the editors DATE.
\funding{The research of Geng Chen was supported in part by the US National Science Foundation Grant DMS-1715012.
The research of Gui-Qiang G. Chen was supported in part by
the UK
Engineering and Physical Sciences Research Council Awards
EP/L015811/1 and EP/V008854,
and the Royal Society--Wolfson Research
Merit Award WM090014 (UK). The third author was supported in part by the Australian Research Council under grant DP170100630, the National Natural Science Foundation of China under Grant 12101395, the Royal Society-Newton International Fellowships NF170015, the Newton International Fellowships Alumni AL/201021 and AL/211005, and Shanghai Frontier Science Research Center for Modern Analysis.}}}
\author{Geng Chen \thanks{ School of Mathematics, University of Kansas, Lawrence, KS 66045, USA
  (\email{gengchen@ku.edu}).}
\and Gui-Qiang G. Chen \thanks{ Mathematical Institute, University of Oxford, Oxford  OX2 6GG, UK
  (\email{chengq@maths.ox.ac.uk}).}
\and Shengguo Zhu \thanks{ School of Mathematical Sciences, and MOE-LSC, Shanghai Jiao Tong University, Shanghai 200240, China
  (\email{zhushengguo@sjtu.edu.cn}).}}

\usepackage{amsopn}

\ifpdf
\hypersetup{
  pdftitle={Formation of Singularities
and Existence of \\  Global Continuous Solutions  \\ for the Compressible Euler Equations},
  pdfauthor={G. Chen, G.-Q.  Chen,  and S. Zhu}
}
\fi

\begin{document}

\maketitle

\begin{abstract}
We are concerned with the formation of singularities and the existence of global continuous
solutions of the Cauchy problem for the one-dimensional nonisentropic Euler
equations for compressible fluids.
For the isentropic Euler equations, we pinpoint a necessary and sufficient
condition for the formation of singularities of solutions with large initial data that allow
a far-field vacuum --- there exists a compression in the initial data.
For the nonisentropic Euler equations, we identify a sufficient condition for the formation of singularities
of solutions with large initial data that allow a far-field
vacuum --- there exists a strong compression in the initial data.
Furthermore, we identify two new phenomena --- decompression and de-rarefaction --- for the
nonisentropic Euler flows, different from the isentropic flows, via constructing two respective solutions.
For the decompression phenomenon, we construct a first global continuous nonisentropic solution, even though
the initial data contain a weak compression, by solving a backward Goursat problem, so that
the solution is smooth, except on several characteristic curves across which the solution has a weak
discontinuity ({\it i.e.}, only Lipschitz continuity).
For the de-rarefaction phenomenon, we construct a continuous nonisentropic solution whose initial data
contain isentropic rarefactions ({\it i.e.}, without compression)
and  a locally stationary varying entropy profile,
for which the solution still forms a shock wave in a finite time.
\end{abstract}



\section{Introduction}
The Euler equations are widely used to describe inviscid compressible fluid flows, particularly in gas dynamics,
which consist of the conservation laws of mass, momentum, and energy in general.
The nonisentropic Euler equations in Lagrangian coordinates $(t,x)\in\mathbb{R}^+\times \mathbb{R}$
in one space-dimension take the following form:
\begin{align}
&v_t-u_x =0,\label{lagrangian1q}\\
&u_t+p_x=0,\label{lagrangian2q}\\
&\textstyle\big(e+\frac{1}{2}u^2\big)_t+(u\,p)_x=0, \label{lagrangian3q}
\end{align}
where $\rho$ is the density, $v=\rho^{-1}$ the specific volume,
$p$ the pressure, $u$ the velocity, and $e$ the specific  internal energy.
The system is closed by the constitutive relations governed by
\beq
T\,\dd S = \dd e + p\,\dd v,
\label{2TD}
\eeq
where $S$ is the entropy and $T$ the temperature.
The solutions ($C^1$ or weak) for the compressible Euler equations
in Lagrangian and Eulerian coordinates are equivalent;
see \cite{Chen1992,Dafermos2010,Wagner1987}.

For clarity of presentation, in this paper, we focus on the case that the gas is ideal polytropic
so that
\[
 p\, v=R\,T, \qquad  e=c_v\,T
\]
with ideal gas constant $R$ and specific heat $c_v$. This implies
\beq
 p=K\,e^{\frac{S}{c_v}}\, v^{-\gamma},\quad  e=\frac{p v}{\gamma-1}\,,
\label{introduction 3}
\eeq
with adiabatic gas constant $\gamma>1$, and
$K$ and $c_v$ as positive constants ({\it cf}. \cite{courant,smoller}).

In particular, for $C^1$ solutions,
(\ref{lagrangian3q}) is equivalent to the conservation of entropy \cite{smoller}:
\beq\label{s con}
 S_t=0,
\eeq
so that, for any $x\in \mathbb{R}$,
\[
S(t,x)\equiv S(0,x)=:S(x).
\]

If entropy $S$ is a constant, then the flow is isentropic, and
equations (\ref{lagrangian1q})--(\ref{lagrangian2q}) themselves become a closed system,
known as the $p$--system:
\begin{align}
&v_t-u_x=0,\label{p1}\\
&u_t+p_x=0, \label{p2}
\end{align}
with
\beq\label{p3}
   p=K\,v^{-\gamma} \qquad \mbox{for $\gamma>1$},
\eeq
where, without loss of generality, we still use $K$ to denote the constant in pressure.

The compressible Euler equations are the most fundamental prototype
of hyperbolic conservation laws:
\beq
  {\bf u}_t+{\bf f}({\bf u})_x=0, \qquad
  {\bf u}\in\mathbb{R}^n,\,\,\,  {\bf f}:\mathbb{R}^n\to\mathbb{R}^n.
\label{conservation laws}
\eeq
The solutions of \eqref{conservation laws}, even starting out with smooth initial data,
often form physical discontinuities such as shock waves in a finite time.
Then the mathematical analysis of the solutions becomes delicate, owing
to the lack of regularity when the singularity forms.
In this paper, our analysis is focused on the formation of singularities
and the existence of singularity-free solutions
for both the isentropic Euler equations with large initial data
that allow a far-field vacuum
and the nonisentropic Euler equations with nontrivial large initial data.
In particular, we identify new phenomena --- decompression and de-rarefaction --- for
the nonisentropic Euler flows, different from the isentropic flows,
via constructing two respective solutions.

The study of breakdown of $C^1$ solutions for \eqref{conservation laws}
has a considerable history; see \cite{Dafermos2010} and the references therein.
The breakdown for scalar equations goes back to Stokes \cite{Stokes}, and was
solved for $2\times 2$ systems by Lax~\cite{lax2} when the solution
is in a strictly hyperbolic region.
These results state that, in the presence of genuine nonlinearity,
nontrivial initial data (even with small oscillation)
often lead to the blowup of the gradient of the solution in a finite time.
For larger systems, similar results are available,
provided that the initial data are
small ({\it cf}. \cite{Fritzjohn,Lidaqian,liu}).

For the initial data of large oscillation uniformly away from the vacuum for the compressible Euler equations,
a fairly good understanding of the formation of singularities based on the fundamental
work of Lax \cite{lax2} has been obtained; see \cite{G3,G9, CPZ,G8}.

First, for the isentropic Euler equations (the $p$--system), the Riccati equations
established by Lax \cite{lax2} provide a base point for the analysis of the
formation of singularities,
which  lead to an equivalent condition
when $\gamma\geq 3$, together with the existence theory of $C^1$ solutions.
This result states that
{\it the $C^1$ solution breaks down if and only if the non-vacuum initial data contain a compression.}
On the other hand, in order to obtain the same result when $1<\gamma<3$,
it requires a sufficiently good time-dependent lower bound estimate of the density.
In fact, when the density approaches zero as $t$ is large,
which may happen even when the initial density is uniformly away from zero,
the coefficients of the quadratic term in the Riccati equations may approach zero.
The time-dependent lower bound estimates of the density, obtained in \cite{G9,CPZ,CPZ2}
under the assumption
that the initial density has a constant positive lower bound,
 lead to the claim of the same equivalent initial condition
on the result for the formation of singularities for the case $1<\gamma<3$,
as for the case $\gamma\geq 3$.

For the nonisentropic Euler equations,
a similar Riccati system in \cite{G3} was used to give
{\it a sufficient condition on the formation of singularities for the $C^1$ nonisentropic solution
with the initial data that are uniformly away from the vacuum}.
To achieve this, it also requires the constant upper bound estimate of the density  and velocity ({\it cf}. \cite{G8})
and the time-dependent lower bound of the density ({\it cf}. \cite{G9,CPZ}).

However, there are still several fundamental open issues without clear answers yet
for the large data problems in the one-dimensional case, especially when the large
initial data allow a far-field vacuum.
In this paper, we establish a fairly complete theory on the formation of singularities
and the existence of nontrivial singularity-free solutions.
More precisely, we address the following three issues:

\smallskip
\paragraph{{\rm (i)}.  The formation of singularities of solutions with initial data that allow
a far-field vacuum}
For all known results of the formation of singularities in \cite{G9,CPZ} for both the isentropic and nonisentropic Euler equations
with large initial data,
it requires a positive lower bound of the initial density,
which leads to a time-dependent lower bound of the  density for later time.
In this paper, we provide a new method to extend the theory to more general initial density profiles
that allow a far-field vacuum such as $\rho(0,x)\in (C^1\cap L^1)(\mathbb{R})$: For the isentropic case,
we pinpoint a necessary and sufficient
condition for the formation of singularities --- the initial data contain a compression;
while, for the nonisentropic case,
we identify a sufficient condition for the formation of
singularities --- the initial data contain a {\it strong} compression.

\smallskip
\paragraph{{\rm (ii)}.  The phenomenon of decompression}
Different from the isentropic flow, a {\it weak} compression for the nonisentropic flow can be cancelled
during the wave interactions;
see also the examples including contact discontinuities and compression waves ({\it cf}. \cite{G6,youngblake1}).
An open question has been whether there exists a global $C^1$ solution including
an initial {\it weak} compression for the nonisentropic case.
In this paper, we answer this question and provide a first global continuous
solution for the nonisentropic case
when the initial data contain the {\it weak} compression,
by developing a method via solving a backward Goursat problem.
The solution is smooth, except on several characteristics across which the solution has a weak discontinuity
({\it i.e.}, only Lipschitz continuity).
The method developed here for the construction is new, which provides a new approach to the constructions
of other solutions with similar features.

\smallskip
\paragraph{{\rm (iii)}.  The phenomenon of  de-rarefaction}
We succeed to construct a first continuous solution
with nonisentropic rarefactive initial data, so that the solution forms a shock wave in a finite time.
More precisely, the initial data include a pair of forward and backward isentropic rarefaction simple
waves and a locally stationary solution with varying entropy.

Moreover, in order to show what types of singularities the solution may form after the breakdown,
we revisit some earlier results
and show that, under some additional assumptions,
all the singularities in these previous results
must be shock waves.

\smallskip
The formation of singularities of solutions has also been studied for the compressible Euler equations
and systems of hyperbolic conservation laws in multidimensional space variables under various assumptions
on the initial data;
see \cite{ Chris, shuang, Dafermos2010, Yachun,  Speck, Makino, Rammaha, Sideris} and the references therein.

\smallskip
This paper is divided into six sections and one appendix.
In \S 2, we introduce some basic notations and necessary equations.
In \S 3, we establish the theorems for the singularity formation  of the $C^1$ solution with large data
for the compressible Euler equations  allowing a far-field vacuum.
In \S 4, we present a new phenomenon -- decompression -- of the nonisentropic Euler flow
by constructing a global continuous solution even starting from the initial data that contain a weak compression.
In \S 5, we construct a continuous solution with nonisentropic rarefactive initial data, so that the solution forms a shock wave
in a finite time.
In \S 6, we give two additional remarks.
First, under some more restrictive assumption than \eqref{yq1} in Theorem \ref{p_sing_thm} on the initial far-field density,
we give a better estimate on the lower bound of the density than that of \S 3.
Then we  give a remark to show that, under some proper assumptions,
the singularity caused by the initial compression is actually a shock wave.
Finally, we give an appendix to present a local-in-time well-posedness theorem in some angular domains,
which is used in the proofs in \S 4--\S5.

\section{Basic setup}
In this section, we provide some basic notations, equations, and estimates for $C^1$ solutions
of the Euler equations \eqref{lagrangian1q}--\eqref{lagrangian3q} with initial data:
\beq
(v,u,S)|_{t=0}=(v_0, u_0,S_0)(x) \label{initial1}
\eeq
for subsequent development.
The results in this section have been basically established in \cite{G9, CPZ,lax2},
with some slight modification and improvement.

To make our statement self-contained,
we first introduce the notion of $C^1$ solutions.

\begin{definition}\label{ 2.1-1}
Let $T>0$. A vector function $(v, u, S)(t,x)$
is called
a $C^1$ solution of the Euler equations \eqref{lagrangian1q}--\eqref{lagrangian3q}
on $(0, T ) \times \mathbb{R}$  if
\begin{equation*}
v>0, \quad (v,u,S) \in \big(C^1([0, T )\times \mathbb{R})\big)^3,
\end{equation*}
and equations \eqref{lagrangian1q}--\eqref{lagrangian3q} are satisfied pointwise
on $(0,T) \times \mathbb{R}$.
It is called a $C^1$ solution of the Cauchy problem \eqref{lagrangian1q}--\eqref{lagrangian3q} and \eqref{initial1}
if it is a $C^1$ solution of equations \eqref{lagrangian1q}--\eqref{lagrangian3q} on $(0, T ) \times \mathbb{R}$
and satisfies its initial data  \eqref{initial1} in $C^0$.
\end{definition}

\begin{remark}
Let $(v, u, S)$ be a $C^1$ solution of the Euler equations \eqref{lagrangian1q}--\eqref{lagrangian3q}
on $(0, T ) \times \mathbb{R}$.
Then
$v$ satisfies
$$
0<v<\infty, \qquad v \in C^1([0, T )\times \mathbb{R}).
$$
Since $0<v(t,x)<\infty$ for any $(t,x)\in [0,T]\times \mathbb{R}$,
the Lagrangian coordinate transformation is always invertible
for any $(t,x)\in [0,T]\times \mathbb{R}$.
\end{remark}

Denote
\begin{align}
&\hat{S}:=e^{\frac{S}{2c_v}}>0, \quad
c:=\sqrt{-p_v}=\sqrt{K\,\gamma}\,v^{-\frac{\gamma+1}{2}}\hat{S},
  \label{m def} \\
&a := \int^\infty_v \frac{c}{\hat{S}}\,{\dd}v
         = \TS\frac{2\sqrt{K\gamma}}{\gamma-1}\,
v^{-\frac{\gamma-1}{2}}>0,\label{z def}
\end{align}
where $c$ is the
sound speed.
Then a direct calculation shows
\beq\label{tau p c}
  v=K_{v}\,a^{-\frac{2}{\gamma-1}},\qquad
  p=K_p\, \hat{S}^2\, a^{\frac{2\gamma}{\gamma-1}},\qquad
  c=c(a,\hat{S})=K_c\, \hat{S}\, a^{\frac{\gamma+1}{\gamma-1}},
\eeq
with positive constants
\beq
 K_v:=\Big(\frac{2\sqrt{K\gamma}}{\gamma-1}\Big)^\frac{2}{\gamma-1}\,,
\quad
 K_p:=K\,K_v^{-\gamma},
\quad  K_c:=\TS\sqrt{K\gamma}\,K_v^{-\frac{\gamma+1}{2}}.
\label{Kdefs}
\eeq
Then, for $C^1$ solutions, problem (\ref{lagrangian1q})--(\ref{lagrangian3q}) and (\ref{initial1}) is equivalent to
\begin{equation}\label{fulleuler1}
\begin{cases}
  a_t+\frac{c}{\hat{S}}\,u_x=0, \\[4pt]
  u_t+\hat{S}\,c\,a_x+\frac{2p}{\hat{S}}\,\hat{S}_x=0,\\[3pt]
  \hat{S}_t=0,\\[3pt]
(a, u, \hat{S})(x, 0)=(a_0, u_0, \hat{S}_0)(x)=(a(v_0(x)), u_0(x), e^{\frac{S_0(x)}{2c_v}}).
\end{cases}
\end{equation}
Note that $\hat{S}=\hat{S}(x)=\hat{S}_0(x)$ is independent of $t$.

We denote the directional derivatives as
\[
\pp := \dbyd t+c\,\dbyd x, \qquad   \pn := \dbyd t-c\,\dbyd x\,,
\]
along the corresponding two characteristic directions:
\beq\label{pmc_full}
  \frac{\dd x^+}{\dd t}=c, \qquad \frac{\dd x^-}{\dd t}=-c.
\eeq
Denote the corresponding Riemann variables:
\beq
   s:=u+\hat{S}\,a,\qquad r:=u-\hat{S}\,a.
\label{r_s_def}
\eeq
Then $r$ and $s$ satisfy
\begin{align}
 \pp s=
 \pn r=\frac{c}{2\gamma}\,\frac{\hat{S}_x}{\hat{S}}\,(s-r).
\label{s_eqn}
\end{align}

\subsection{Isentropic case}
Now we introduce the Riccati equations for the gradient variables,
which are the basis for our proof of the formation of singularities.
We start with the isentropic case.

Without loss of generality, for the isentropic case, we always assume
\[
\hat{S}(x)=1,
\]
that is,
\beq
(v,u,S)|_{t=0}=(v_0(x), u_0(x), 0).
\label{initial1.2}
\eeq
For this case, the Riemann invariants:
\beq\label{p_sr}
s=u+a,  \qquad r=u-a,
\eeq
with
$a$ defined in \eqref{z def}, are constant
along the forward and backward characteristics, respectively:
\beq\label{srconq}
 \partial_+s=0, \qquad  \partial_- r=0.
\eeq
Furthermore, the gradient variables:
\beq\label{albeis}
s_x=u_x+a_x,\qquad r_x=u_x-a_x
\eeq
can be used to describe the compression and rarefaction phenomena of nonlinear waves.

It is observed in Lax \cite{lax2} that, by adding an integrating factor to both $s_x$ and $r_x$
to form  a new pair of gradient variables $\alpha$ and $\beta$:
\beq\label{albeis_yq}
\alpha:= a^{\frac{\gamma+1}{2(\gamma-1)}}s_x, \qquad  \beta:= a^{\frac{\gamma+1}{2(\gamma-1)}}r_x,
\eeq
then $\alpha$ and $\beta$ satisfy two ``decoupled'' Riccati equations:
\begin{align}\label{p_y_eq-1}
  \partial_+ \alpha = - b \, \alpha^2, \qquad
  \partial_- \beta = - b \, \beta^2,
  \end{align}
where
\begin{align}\label{a2}
  b =b(a)&:= \hat{K}\,{\TS\frac{\gamma+1}{2(\gamma-1)}}\,
        a^{\frac{3-\gamma}{2(\gamma-1)}}=\tilde{K} \rho^{\frac{3-\gamma}{4}}
\end{align}
with positive constants $\hat{K}$ and $\tilde{K}$.
Here we notice that the integrating factor $a^{\frac{\gamma+1}{2(\gamma-1)}}$
is the major term in $\sqrt{c}$, where $c$ is the sound speed given by
\begin{equation}\label{localsound}
c=\sqrt{K \gamma}\, \rho^{\frac{\gamma+1}{2}}.
\end{equation}

When $1<\gamma<3$,
in order to show the formation of singularities in a finite time
which is equivalent to there existing some initial
compression:
$$
\inf_{x\in \mathbb{R}}\{\alpha,\beta\}(0,x)<0,
\quad\mbox{or equivalently}\quad \inf_{x\in \mathbb{R}}\{s_x, r_x\}(0,x)<0,
$$
it requires some sufficiently good time-dependent lower bound of the density:
\beq\label{smc}
\int_0^\infty b(t, x(t))\,\dd t= \tilde{K}\int_0^\infty \rho^{\frac{3-\gamma}{4}}(t, x(t))\,\dd t=\infty
\eeq
along any forward or backward characteristic $x(t)$.
Note that the example with density decaying to zero as $t\rightarrow \infty$
does exist in some cases \cite{courant}, so that it is impossible to find
a positive time-independent lower bound of the density in general.

When the initial density has a uniform positive lower bound,
the time-dependent lower bound of the density in its optimal order
has helped the proof of \eqref{smc} in \cite{G9,CPZ}.
In this paper, we present an approach to establish the necessary and sufficient result
for the formation of singularities, even for the initial data
that allow a far-field vacuum.

Finally, we define the {\it compression and rarefaction
characters}
for the isentropic case
which generalize the definition of compression and rarefaction simple waves
to the local sense.

\begin{definition}
\label{def1}
The local compression/rarefaction characters for a $C^1$ solution of the isentropic Euler equations \eqref{p1}--\eqref{p2} are
\[
\begin{array}{lllll}
	\text{Forward rarefaction} &\qquad \text{iff}& s_x>0 &\qquad \text{iff} \quad s_t< 0,\\
	\text{Forward compression} &\qquad \text{iff} & s_x<0 &\qquad \text{iff} \quad s_t> 0,\\
	\text{Backward rarefaction} &\qquad \text{iff}& r_x>0&\qquad \text{iff} \quad r_t> 0,\\
	\text{Backward compression} &\qquad\text{iff}&r_x<0&\qquad \text{iff} \quad r_t< 0.
\end{array}
\]
\end{definition}

Although this definition was not exactly stated in Lax \cite{lax2},
the result for some cases of $2\times2$ hyperbolic conservation laws
can be explained as follows: {\it The singularity forms in a finite time if and only if
there exists some backward or forward compression in the sense of Definition {\rm \ref{def1}}}.
Thus, this definition of the compression and rarefaction
gives a clean cut for the formation of singularities.

\subsection{Nonisentropic case}
We now review the Riccati equations for the nonisentropic Euler equations (see also \cite{G3, linliuyang}).
Define
\begin{align}
 &\xi:=-\frac{s_t}{c}=s_x-\frac{1}{\gamma}\hat{S}_x a=
   u_x+\hat{S}a_x+\frac{\gamma-1}{\gamma}\,\hat{S}_x\, a, \label{def-alpha}\\
&\zeta :=\frac{r_t}{c}=r_x+\frac{1}{\gamma}\hat{S}_x a=
   u_x-\hat{S}a_x-\frac{\gamma-1}{\gamma}\,\hat{S}_x\, a,
\label{def_beta}
\end{align}
which are consistent with the definition in \eqref{albeis} for isentropic solutions when $\hat{S}=1$.
In fact, more intuitively, the definition of $(\xi,\zeta)$
can be equivalently given by the  following lemma (see also \cite{G3,G6}).

\begin{lemma}\label{riclemma}
Any $C^1$ solution
of \eqref{lagrangian1q}--\eqref{lagrangian3q} satisfies
\beq
\displaystyle
\xi=-\frac{s_t}{c}=-\frac{\partial_-{u}}{c}=-\frac{\partial_-{p}}{c^2},
\qquad
\zeta=\frac{r_t}{c}=\frac{\partial_+{u}}{c}=-\frac{\partial_+{p}}{c^2},\label{forward}
\eeq
and
\beq
\partial_+\xi=k_1\big(k_2 (3\xi+\zeta)+\xi\zeta-\xi^2\big),
\qquad
\partial_-\zeta=k_1\big(-k_2 (\xi+3\zeta)+\xi\zeta-\xi^2\big),
\label{frem1}
\eeq
where
\beq
k_1=\frac{(\gamma+1)K_c}{2(\gamma-1)} a^{\frac{2}{\gamma-1}}, \qquad
k_2=\frac{\gamma-1}{\gamma(\gamma+1)}a\, \hat{S}_x. \label{k def}\
\eeq
\end{lemma}

In order to transform (\ref{frem1}) into ``decoupled'' equations,
we introduce new variables $(\alpha, \beta)$ transformed
from $(\xi, \zeta)$, analogous to the isentropic solutions:
\begin{align}
  \alpha &:= \hat{S}^{-\frac{3(3-\gamma)}{2(3\gamma-1)}}\,
       a^{\frac{\gamma+1}{2(\gamma-1)}}\,
       \big(s_x - {\TS\frac{2}{3\gamma-1}}\,\hat{S}_x\,a\big)
       =\hat{S}^{-\frac{3(3-\gamma)}{2(3\gamma-1)}}\,
       a^{\frac{\gamma+1}{2(\gamma-1)}}\,
       \big(\xi+ {\TS\frac{\gamma-1}{\gamma(3\gamma-1)}}\,\hat{S}_x\,a\big), \nn\\
  \beta &:= \hat{S}^{-\frac{3(3-\gamma)}{2(3\gamma-1)}}\,
       a^{\frac{\gamma+1}{2(\gamma-1)}}\,
       \big(r_x + {\TS\frac{2}{3\gamma-1}}\,\hat{S}_x\,a\big)
        =\hat{S}^{-\frac{3(3-\gamma)}{2(3\gamma-1)}}\,
       a^{\frac{\gamma+1}{2(\gamma-1)}}\,
       \big(\zeta- {\TS\frac{\gamma-1}{\gamma(3\gamma-1)}}\,\hat{S}_x\,a\big).
       \label{intr mainq}
\end{align}
Then, by some calculations ({\it cf}. \cite{G3}),  $(\alpha, \beta)$ satisfy
\begin{align}
  \partial_+ \alpha = b_0- b_2 \, \alpha^2, \qquad  \partial_- \beta = b_0- b_2 \, \beta^2,
\label{yq odesq}
\end{align}
where
\begin{align}
  {b}_0 &:= {\TS \frac{(\gamma-1)K_c}{\gamma(3\gamma-1)}} \,
        \big(\TS\,\hat{S}\,\hat{S}_{xx}
         - {\TS\frac{3\gamma+1}{3\gamma-1}}\,\hat{S}_x^2\big)\,
	\hat{S}^{-\frac{3(3-\gamma)}{2(3\gamma-1)}}\,
         a^{\frac{3(\gamma+1)}{2(\gamma-1)}+1},\nn\\
  {b}_2 &:= {\TS\frac{(\gamma+1)K_c}{2(\gamma-1)}}\,
	\hat{S}^{\frac{3(3-\gamma)}{2(3\gamma-1)}}\,
        a^{\frac{3-\gamma}{2(\gamma-1)}}.
\label{adefsq}
\end{align}

\begin{remark}
Owing to the complex structure of the nonisentropic Euler equations,
it is not easy to give a clear definition of rarefaction and compression characters
as in the isentropic case.
Notice that the stationary $($Lagrangian$)$ solutions
with varying entropy involve no rarefaction or compression.
Thus, the signs of $s_x$ and $r_x$ {\rm (}resp., $\alpha$ and $\beta${\rm )} cannot be used
to determine the rarefaction and compression characters of the nonisentropic solution in general,
since $(s_x,r_x)$  $($resp., $(\alpha,\beta)${\rm )}  might
be nonzero in the stationary $($Lagrangian$)$ solution with varying entropy.
On the other hand,
since $(u, p)$ are Riemann invariants for the second characteristic family,
it follows  from \eqref{lagrangian1q}--\eqref{lagrangian3q} that
$(u,p)$ are constant functions in the stationary solution with smoothly varying entropy.
Then, by Lemma {\rm \ref{riclemma}}, we see that $(s_t, r_t)$ or $(\xi,\zeta)$ are also zero
in the stationary solution with smoothly varying entropy.
This is one of the reasons why we use $(s_t,r_t)$ or $(\xi,\zeta)$
in the analysis of the corresponding formation of singularities and well-posedness problems.
\end{remark}

\section{Formation of Singularities for the Initial Data with a Far-Field Vacuum}

This section is devoted to the analysis of the formation of singularities of $C^1$ solutions
with large initial data that allow a far-field vacuum for the compressible Euler equations.
We first state our  main results.

\subsection{Main results}

We first focus on the isentropic case \eqref{p1}--\eqref{p2}.
When the initial density $\rho_0$ has a constant positive lower bound:
\beq\label{exist}
\max_{x\in R}v_0(x)< \infty,
\eeq
the result of singularity formation by Lax \cite{lax2} has been extended
to include equations \eqref{p1}--\eqref{p2} with any $\gamma>1$
in \cite{G9,CPZ}.
In fact, the result in \cite{lax2} applies for the case of large initial data with $\gamma\geq 3$.
On the other hand, for the case $1<\gamma<3$, it requires a good enough time-dependent lower bound of the density.
Such a lower bound was first given in \cite{CPZ}, and then extended to its optimal order $O((1+t)^{-1})$ in \cite{G9,CPZ2}.
Intuitively, the result can be stated as follows.

\begin{proposition}\label{prop1}
When \eqref{exist} is satisfied, the $C^1$ solution of \eqref{p1}--\eqref{p2} breaks down
if and only if the initial data contain compression somewhere in the sense of Definition {\rm\ref{def1}}.
\end{proposition}

The global existence theorem in this result can be obtained
by the existence theory of $C^1$ solutions ({\it cf}. \cite{Lidaqian}).
This can also be seen from \eqref{case1} and \eqref{4.6a} in the proof of Theorem \ref{p_sing_thm} in \S 3.2.

\vspace{.2cm}
A natural question is whether Proposition \ref{prop1} still holds when \eqref{exist} is not satisfied.
This is not direct. For example, when
\[
0<\rho_0(x)\in (C^1\cap L^1)(\mathbb{R}),
\]
and
\beq\label{vff}
\lim_{x\rightarrow -\infty}\rho_0(x)=0\qquad\hbox{or}\qquad \lim_{x\rightarrow \infty}\rho_0(x)=0,
\eeq
the existing methods in  \cite{G9,CPZ,lax2} do not work for a proof of the equivalent condition
for the formation of singularities when $1<\gamma<3$, since there is lack of a global lower bound estimate
of the density.
A further question for nonisentropic solutions is to see whether the results in \cite{G9,CPZ}
can be extended to the case when the initial density
allows a far-field vacuum.
In this paper, we provide a new method to tackle these questions.
Instead of the initial assumptions used in \cite{G9,CPZ} including \eqref{exist},
we assume that the initial data satisfy the following conditions:

\begin{condition}\label{initialdata2}
The initial data $(\rho_0, u_0, S_0)=(v_0^{-1}, u_0, S_0)$ satisfy
\begin{equation}\label{initial3}
\begin{split}
&\rho_0>0,\quad  (\rho_0,u_0)\in C^1(\mathbb{R}), \quad S_0\in C^2(\mathbb{R}),\\
&\|(\rho_0,u_0)\|_{C^1(\mathbb{R})}+\|S_0\|_{C^2(\mathbb{R})}=K_0\in (0, \infty).
\end{split}
\end{equation}
Furthermore, for  $1<\gamma<\frac{5}{3}$,  we assume
\begin{equation}\label{yq1}
Y=\max\big\{0,\ \sup_x\,\alpha(0,x)\big\}<\infty, \qquad  Q= \max\big\{0,\ \sup_x\,\beta(0,x)\big\}<\infty,
\end{equation}
where $\alpha$ and $\beta$ are defined in  \eqref{intr mainq}.
\end{condition}

\begin{condition}\label{bv}
Assume that the initial entropy $S_0(x)\in C^2$
has finite total variation so that
\beq
  V := \frac{1}{2c_v}\int_{-\infty}^{\infty}|S_0'(x)|\;\dd x
     = \int_{-\infty}^{\infty}\frac{|\hat{S}_0'(x)|}{\hat{S}_0(x)}\;\dd x<\infty.
\label{Vdef}
\eeq
\end{condition}

\begin{remark}
The  above initial conditions allow the appearance of vacuum in the far-field, {\it i.e.}, \eqref{vff}.
Condition {\rm \eqref{initial3}} implies that,
for $\gamma\geq \frac{5}{3}$, \eqref{yq1} still holds.
For the isentropic case, $S_0(x)=\bar S$ is a constant so that Condition {\rm \ref{bv}}
is automatically satisfied.
\end{remark}

Now we state our main results for the isentropic and nonisentropic solutions,
respectively.

\begin{theorem}\label{p_sing_thm}
Consider the isentropic solutions with $S_0(x)=\bar S$.
For $\gamma>1$, if $(\rho_0(x), u_0(x), \bar S)$ satisfy  Condition {\rm \ref{initialdata2}},
then the Cauchy problem  \eqref{p1}--\eqref{p2} with   \eqref{initial1.2} has a unique
global-in-time $C^1$ solution if and only if the initial data satisfy
\beq\label{p_lemma_con}
\min\big\{\inf_{x\in \mathbb{R}}\alpha(0,x), \inf_{x\in \mathbb{R}}\beta(0,x)\big\}\geq 0.
\eeq
\end{theorem}

\begin{theorem}\label{Thm singularity2}
Consider the $C^1$ solution of the Cauchy problem \eqref{lagrangian1q}--\eqref{lagrangian3q}
with \eqref{initial1},
where $\gamma>1$ and the initial data $(\rho_0(x), u_0(x), S_0(x))$ satisfy
Conditions {\rm\ref{initialdata2}}--{\rm\ref{bv}}.
Then there exists a positive constant $N$ depending only on $(K, c_v)$ and the initial data
{\rm (}which is precisely given in \eqref{Ndefq} later{\rm )} such that,
if the initial data satisfy the strong compression condition{\rm :}
\beq
\min\big\{\inf_{x\in\mathbb{R}} \alpha(0,x), \inf_{x\in \mathbb{R}} \beta(0,x)\big\} < -N,
\label{yq-N}
\eeq
then  $|\rho_x|$ and/or $|u_x|$ must blow up in a finite time.
\end{theorem}

It should be pointed out that $\eqref{yq-N}$ is a sufficient condition for the formation of singularities,
which has not included all of the initial data with the weak compression, as shown by the decompression
phenomenon justified through the decompression example in \S 4.
Now we first give two examples to show that either the finite-time breakdown or global well-posedness
can happen for $C^1$ solutions of the Cauchy problem \eqref{lagrangian1q}--\eqref{lagrangian3q} with \eqref{initial1}
when $\min\{\inf_x \alpha(0,x),\inf_x \beta(0,x)\}\geq 0$.

\begin{example}\label{Example-3.1}
Consider the $C^1$ solution $(\rho, u,S)$ of the Cauchy problem \eqref{lagrangian1q}--\eqref{lagrangian3q} and \eqref{initial1}
with the following initial data{\rm :}
\begin{equation}\label{klnm}
S_0(x)=\arctan (x), \quad a_0(x)=\hat{S}^{-\frac{3(\gamma-1)}{3\gamma-1}}_0,\quad u_0=\epsilon \arctan (x),
\end{equation}
where $\epsilon>0$ is a sufficiently small parameter.
It follows from a direct calculation that
\begin{equation}\label{klnm-2}
\begin{split}
\alpha(0,x)=&\,\beta(0,x)=\frac{\epsilon}{1+x^2} e^{-\frac{3\arctan (x)}{(3\gamma-1)c_v}}>0,\\[4pt]
\xi(0,x)=&\,\Big(\epsilon^{}-\frac{\gamma-1}{2\gamma(3\gamma-1)c_v}\hat{S}^{\frac{2}{3\gamma-1}}_0\Big)\frac{1}{1+x^2}<0,\\[4pt]
\zeta(0,x)=&\,\Big(\epsilon^{}+\frac{\gamma-1}{2\gamma(3\gamma-1)c_v}\hat{S}^{\frac{2}{3\gamma-1}}_0\Big)\frac{1}{1+x^2}>0
\end{split}
\end{equation}
for sufficiently small $\epsilon>0$ and $x\in \mathbb{R}$,
which implies that $(\rho,u,S)$ does not satisfy our condition $(\ref{yq-N})$.
This example was first considered in Pan-Zhu {\rm \cite{huahua}},
in which  the finite-time breakdown was proved when $1<\gamma<3$ with sufficiently small  $\epsilon>0$.
\end{example}

\begin{example}\label{Example-3.2}
For the initial-boundary value problem  $($IBVP$)$ of \eqref{lagrangian1q}--\eqref{lagrangian3q}
in $[0,\infty)\times [0,\infty)$, Lin-Liu-Yang {\rm \cite{liu}} proved that,
if the $C^1$ smooth initial-boundary data $($away from the vacuum$)$ satisfy
\begin{equation}\label{klnm-3}
\begin{split}
&\alpha(0,x)\geq 0,\quad \beta(0,x)\geq 0 \qquad\,\,\, \text{for $x\geq  0$},\\
&\alpha(t,0)\geq 0\qquad\,\,\, \text{for $t\geq 0$}, \\
&  \|S_x\|_{L^1}\leq  C_0,   \quad S\in C^2([0, \infty)), \\
& S_{xx}- (3\gamma-1)c_v S^2_x >0 \qquad\,\,\, \text{for $x\geq  0$}\\
\end{split}
\end{equation}
for some constant $C_0>0$ depending on the lower bound of $\rho_0$,
then there is a unique global-in-time  $C^1$ solution of the corresponding IBVP.
\end{example}

In \S 3.2--\S 3.3 below, we give the proofs of Theorems \ref{p_sing_thm}--\ref{Thm singularity2}, respectively.

\subsection{Proof of Theorem \ref{p_sing_thm}: The $p$--system}
We first review the estimates for the upper and lower bounds of the density and velocity.

First, since the Riemann invariants are constant along the characteristics, as indicated in (\ref{srconq}), for this case,
it is direct to obtain the upper bounds of $\rho$ and $|u|$.

\begin{lemma}\label{densityupperbound}
For the smooth solution, there exists some constant $K_{1}>0$ such that
$$
\|(\rho,u,c,a)\|_{L^\infty([0,T]\times \mathbb{R})}\leq K_{1}<\infty.
$$
\end{lemma}

\smallskip
To obtain the lower bound of the density, we first bound $(\alpha,\beta)$.
In fact,  by the equations in (\ref{p_y_eq-1}),
$(\alpha, \beta)$ decay with respect to $t$ along the characteristics,
so it is also direct to verify the following lemma:

\begin{lemma}\label{lemma_p_2}
If the initial data $(\rho_0(x), u_0(x), \bar S)$ satisfy Condition {\rm\ref{initialdata2}},
then the smooth solution $(\rho, u)(t,x)$  of the Cauchy problem \eqref{p1}--\eqref{p2} with \eqref{initial1.2}
satisfies
\[
\alpha(t,x)\leq Y,  \quad  \beta(t,x)\leq Q\qquad\,\,\, \text{for $(t,x)\in [0,T]\times \mathbb{R}$},
\]
where constants $Y$ and $Q$ are defined in Condition {\rm\ref{initialdata2}}.
\end{lemma}

With Lemma \ref{lemma_p_2}, we are able to prove the key estimate
of the lower bound of the density (equivalently, the upper bound of $v$)
when $1<\gamma<3$.

\begin{lemma}\label{density_low_bound_1-3}
Let $(\rho, u)(t,x)$ be a $C^1$ solution of  the Cauchy problem \eqref{p1}--\eqref{p2} with \eqref{initial1.2}
defined on the time interval $[0, T)$ for some $T>0$,
with the large initial data satisfying Condition {\rm\ref{initialdata2}}.
If $1<\gamma<3$,  then, for any $(t,x)\in  [0, T)\times \mathbb{R}$,
there is a positive constant $K_*$ depending only on $\gamma$ such that
\[
v(t,x)
\le \Big({v_0}^{\frac{3-\gamma}{4}}(x)+K_*(Y+Q) t\Big)^{\frac{4}{3-\gamma}}
\qquad\,\, \text{for $(t,x)\in [0,T)\times \mathbb{R}$}.
\]
\end{lemma}

\begin{proof} From the definition of $(\alpha,\beta)$, we have
$$
\alpha=\sqrt{\frac{c}{K_c}}s_x, \quad \beta= \sqrt{\frac{c}{K_c}}r_x,
$$
which implies
$$
\alpha+\beta=\sqrt{\frac{c}{K_c}} (r_x+s_x)=2\sqrt{\frac{c}{K_c}}u_x.
$$
Therefore, it follows from the mass equation that
$$
\sqrt{c}\,v_t=\frac12 \sqrt{K_c}(\alpha+\beta).
$$
Using the sound speed formula \eqref{localsound} and Lemma \ref{lemma_p_2},
we have
$$
v^{-\frac{\gamma+1}{4}}v_t\le \frac12(K\gamma)^{-\frac14}\sqrt{K_c}(Y+Q).
$$

Notice that $\frac{\gamma+1}{4}<1$ when $1<\gamma<3$.
Then,  for any $x\in \mathbb{R}$ and $t\in [0, T)$,
a simple time integration leads to
$$
v(t,x)\le \Big({v_0}^{\frac{3-\gamma}{4}}(x)+K_*(Y+Q) t\Big)^{\frac{4}{3-\gamma}},
$$
where $K_*=\frac{3-\gamma}{8}(K\gamma)^{-\frac14}\sqrt{K_c}$.
This completes the proof.
\end{proof}

\begin{remark}
If the condition{\rm :}
\beq\label{apenc}
\max\big\{\sup_{x\in \mathbb{R}} s_x(0, x),\, \sup_{x\in \mathbb{R}} r_x(0, x)\big\}<M
\eeq
is further satisfied, then a better estimate in the order of $O((1+t)^{-1})$ can be obtained,
which can be found  in \S {\rm 6.1}.
\end{remark}

\medskip
With Lemmas \ref{densityupperbound}--\ref{density_low_bound_1-3},
we now prove Theorem \ref{p_sing_thm}.

\medskip
\noindent
\emph{Proof of Theorem {\rm \ref{p_sing_thm}}}.
We first note that an equivalent statement of this theorem is that a $C^1$ solution breaks down
in a finite time if and only if
\[
\min\big\{\inf_{x\in \RR}s_x(0,x), \inf_{x\in \RR}r_x(0,x)\big\}<0,
\]
that is, there exists a forward or backward initial compression
in the sense of Definition \ref{def1}.

The global existence of the $C^1$ solution under the condition:
\[
\min\big\{\inf_{x\in \RR}s_x(0,x), \inf_{x\in \RR}r_x(0,x)\big\}\geq 0
\]
is classical, since there exists no vacuum in the initial data
on any finite point $(0,x)$.
We refer the reader to \cite{CPZ} for the details of the proof.

Thus, it suffices to prove the finite-time singularity formation when
\[
\min\big\{\inf_{x\in \RR}s_x(0,x), \inf_{x\in \RR}r_x(0,x)\big\}<0.
\]
Without loss of generality, assume that there exists some $x^*$ such that
\begin{equation}\label{3.13a}
\alpha(0,x^*)<0.
\end{equation}
The proof for the case that $\beta(0,x^{**})<0$ for some
$x^{**}$ is similar, so we omit the details for that case.

Denote
\[
x=\Gamma(t; x^*), \qquad\hbox{or equivalently}\quad t=\Omega(x; x^*),
\]
as the forward characteristic starting from $(0,x^*)$.
From Lemma \ref{density_low_bound_1-3}, we see that, for any fixed $x^*$,
$$
x=\Gamma(t; x^*)=x^*+\int_0^t c(s,\Gamma(s; x^*))\,\text{d}s
$$
is strictly increasing in $t$, and $\Omega(x; x^*)$ is strictly increasing in $x$.
Then
$\lim_{t\rightarrow\infty}\Gamma(t; x^*)$
must exist or equal to infinity,
owing to the monotonicity of $\Gamma(t; x^*)$
for any fixed $x^*$,
which corresponds to the following two cases:

\medskip
\paragraph{\bf Case 1}
$\lim_{t\rightarrow\infty}\Gamma(t; x^*)=\infty$.
A key new idea is to integrate along the $x$--direction, instead of the $t$--direction as done in
\cite{  G9, CPZ, lax2} and the references therein.

From (\ref{p_y_eq-1}) and the definition of $\Gamma(t; x^*)$, we have
\[
\frac{\partial \alpha(\Omega(x; x^*), x)}{\partial x}=- \frac{b}{c} \, \alpha^2
=-K_2\rho^{\frac{1-3\gamma}{4}}\alpha^2,
\]
where
$$
K_2 := {\TS\frac{(\gamma+1)K_C}{2(\gamma-1)\sqrt{\gamma K}}}\,
        \Big(\frac{2\sqrt{K\gamma}}{\gamma-1}\Big)^{\frac{3-\gamma}{2(\gamma-1)}}\, .
$$
Then
\begin{equation}\label{case1}
\alpha(\Omega(x; x^*), x)
=\frac{\alpha(0,x^*)}{1+K_2\alpha(0,x^*)\int_{x^*}^x \rho^{\frac{1-3\gamma}{4}}(\Omega(z; x^*), z)\,\text{d}z}.
\end{equation}

Using Lemma \ref{densityupperbound} and $1-3\gamma<0$ when $\gamma>1$, we have
\begin{equation}\label{inversebound}
\rho^{\frac{1-3\gamma}{4}}\geq K^{\frac{3\gamma-1}{4}}_1 \qquad\,\, \text{for $(t,x)\in [0,T]\times \mathbb{R}$}.
\end{equation}
Thus, there must be some point $x_0>x_*$ such that
$$
1+K_2\alpha(0,x^*)\int_{x^*}^{x_0} \rho^{\frac{1-3\gamma}{4}}(\Omega(z; x^*), z)\,\text{d}z=0
$$
under assumption \eqref{3.13a},
which implies that $\alpha$ or $s_x$ blows up in a finite time.

\smallskip
\medskip
\paragraph{\bf Case 2}
{\it $\lim_{t\rightarrow\infty}\Gamma(t; x^*)=\hbox{L}$
for some finite constant $L>x^*$}.
For this case, we employ the estimate in Lemma \ref{density_low_bound_1-3}.

From (\ref{p_y_eq-1}) and the definition of $\Gamma(t; x^*)$, we have
$$
\frac{1}{\alpha(t,\Gamma(t; x^*))}=\frac{1}{\alpha(0,x^*)}+\int_0^t \, b(\sigma,\Gamma(\sigma; x^*))\;\dd\sigma,
$$
which implies
\begin{equation}\label{4.6a}
\alpha(t,\Gamma(t; x^*))=\frac{\alpha(0,x^*)}{1+y(0,x^*)\int_0^t \, b(\sigma,\Gamma(\sigma; x^*))\;\dd\sigma}.
\end{equation}
It follows  from Lemma \ref{density_low_bound_1-3} and
the definition of $a$ in \eqref{a2} that, when $1<\gamma<3$,
\[
b(t,x)\geq \frac{\ga +1}{4} K_{v}^{-\frac{\ga +1}{4}}\Big(v_0^{\frac{3-\gamma}{4}}(x)+K_*(Y+Q)t\Big)^{-1}
\qquad\,\, \mbox{for $x=\Gamma(t; x^*)\in[x^*,L)$}.
\]
Since $v_0(x)$ with $x=\Gamma(t; x^*)\in[x^*,L)$ is bounded above by some constant depending on $x^*$ and $L$,
but independent of $t$, we have
\[
\int_0^\infty \, b(t,\Gamma(t; x^*))\;\dd t=\infty  \qquad\,\,\mbox{for any $1<\gamma<3$}.
\]
This also holds when $\gamma\geq 3$, since $a$ has a constant lower bound in this case.

Thus, for any $\gamma>1$, there must be some point $t_0>0$ such that
$$
1+\alpha(0,x^*)\int_{0}^{t_0} b(\sigma,\Gamma(\sigma; x^*))\,\text{d}\sigma=0
$$
under assumption \eqref{3.13a},
which shows that $y$ or $s_x$ blows up in a finite time.
This completes the proof.

\subsection{Proof of Theorem \ref{Thm singularity2}: Nonisentropic solutions}

Under Conditions \ref{initialdata2}--\ref{bv} on the initial data,
there are positive constants $M_L$, $M_U$, $M_s$, and $M_r$ such that
\beq
0 < M_L \le \hat{S}(x)=e^{\frac{S_0(x)}{2c_v}} \le M_U, \quad |s_0(x)|\le M_s,\ \quad
|r_0(x)|\le M_r.
\label{m_bounds}
\eeq
Furthermore, denote
\begin{align*}
{\ol V}&:=\frac{V}{2\gamma}=\frac{1}{4\gamma c_v}\int_{-\infty}^\infty |S_0'(x)|\,\dd x,\\
  N_1 &:= M_s+\ol V\,M_r+\ol V\,\big(\ol V\,M_s+{\ol V}^2\,M_r\big)
	\,e^{{\ol V}^2},\\
  N_2 &:= M_r+\ol V\,M_s+\ol V\,\big(\ol V\,M_r+{\ol V}^2\,M_s\big)
	\,e^{{\ol V}^2}.
\end{align*}

In order to handle the solutions of large oscillation,
the first key point is to achieve the $L^\infty$ estimate on the solutions,
which has been fixed by Chen-Young-Zhang \cite{G8}
under Conditions \ref{initialdata2}--\ref{bv} on the initial data,
via a characteristic method.

\begin{lemma}[\cite{G8}]\label{Thm_upper}
Assume the initial data $(\rho_0, u_0, S_0)(x)$ satisfy Conditions {\rm\ref{initialdata2}}--{\rm\ref{bv}}.
If $(\rho(t,x), u(t,x), S(x))$ is a $C^1$ solution of the Cauchy problem \eqref{lagrangian1q}--\eqref{lagrangian3q} and \eqref{initial1}
for $t\in[0,T)$ for some positive $T$, then
\begin{align*}
&|s(t,x)|\le N_1{M_U}^{\frac{1}{2\gamma}}, \qquad |r(t,x)|\le N_2{M_U}^{\frac{1}{2\gamma}},\\
&|u(t,x)|\leq\frac{N_1+N_2}{2}{M_U}^{\frac{1}{2\gamma}}, \qquad  a(t,x)\leq\frac{N_1+N_2}{2}{M_L}^{\frac{1}{2\gamma}-1}:=E_U,
\end{align*}
which implies that there is a positive constant $M_{\rho}$ such that
$$
\rho\le M_{\rho}\qquad \mbox{or equivalently}\quad v\ge \frac{1}{M_{\rho}}.
$$
\end{lemma}

Another key point is how the lower bound estimate of the density can be achieved.
First, similar to Lemma \ref{lemma_p_2}, we are able to find uniform upper bounds for
$(\alpha,\beta)$.
The difference here is that we need to study the structure of
the Riccati system \eqref{yq odesq} in order to obtain such bounds,
which motivates us to study ratio $\frac{b_0}{b_2}$ with $b_0$ and $b_2$ defined in \eqref{adefsq}.
Using this information, we can obtain some lower bound estimate of the density,
as a key lemma to prove the formation of singularities.

We first carry on the calculation on $\frac{b_0}{b_2}$:
\begin{align}\label{a0overa2q}
  {\frac{b_0}{b_2}} =
  {{\TS\frac{2(\gamma-1)^2}{\gamma(\gamma+1)(3\gamma-1)}}\,
        \big(\hat{S}\,\hat{S}_{xx}-{\TS\frac{3\gamma+1}{3\gamma-1}}\,\hat{S}_x^2\big)}\,
    a^{\frac{3\gamma-1}{\gamma-1}}\,\hat{S}^{-\frac{3(3-\gamma)}{3\gamma-1}}.
\end{align}
Note that
\beq \label{bmq}
\hat{S}^2\Big(S_{xx}-\frac{1}{(3\gamma-1)c_v} S_x^2\Big)
={2c_v}\big(\hat{S}\,\hat{S}_{xx}-{\TS\frac{3\gamma+1}{3\gamma-1}}\,\hat{S}_x^2\big),
\eeq
so that $S_{xx}-\frac{1}{c_v(3\gamma-1)} S_x^2$
has the same sign as $b_0$, since $b_2>0$.
Also, we note from the definition of $\hat{S}$ that there is a positive constant $M_3$
such that
\beq\label{m_xxq}
\big|\hat{S}\,\hat{S}_{xx}-{\TS\frac{3\gamma+1}{3\gamma-1}}\,\hat{S}_x^2\big|\le M_3.
\eeq
Define
\beq
  N :=
  \begin{cases}
  \sqrt{\frac{2(\gamma-1)^2}{\gamma(\gamma+1)(3\gamma-1)}\,M_3}
        \  E_U^{\frac{3\gamma-1}{2(\gamma-1)}}
        M_L^{-\frac{3(3-\gamma)}{2(3\gamma-1)}}&\qquad\mbox{for $1<\gamma<3$},\\[2mm]
  \sqrt{\frac{2(\gamma-1)^2}{\gamma(\gamma+1)(3\gamma-1)}\,M_3}
        \,E_U^{\frac{3\gamma-1}{2(\gamma-1)}}
        M_U^{-\frac{3(3-\gamma)}{2(3\gamma-1)}}&\qquad \mbox{for $\gamma\ge 3$}.
  \end{cases}
\label{Ndefq}
\eeq
Then we have
\beq\label{a0a2n2}
\frac{b_0}{b_2}\le  N^2.
\eeq
Since $b_2>0$, by \eqref{yq odesq}, we obtain
\beq\label{a0a2n3}
\partial_+ {\alpha}\leq b_2(N^2-\alpha^2), \qquad \partial_- \beta
\leq b_2(N^2-\beta^2).
\eeq
Following from the comparison theorem for differential equations, we have

\begin{lemma}\label{full_lemma1}
If $(\rho_0, u_0, S_0)(x)$ satisfy Conditions {\rm\ref{initialdata2}}--{\rm\ref{bv}},
then, for any $C^1$ solution $(\rho(t,x), u(t,x), S(x))$ of the Cauchy problem
\eqref{lagrangian1q}--\eqref{lagrangian3q} and \eqref{initial1},
\begin{align*}
\alpha(t,x)\leq\max\big\{N,\ \sup_x \alpha(0,x)\big\}=:\bar Y,
\qquad\beta(t,x)\leq\max\big\{N,\ \sup_x \beta(0,x) \big\}=:\bar Q,
\end{align*}
where $\bar Y$ and $\bar Q$ are two positive constants.
\end{lemma}

The following lemma contains the density lower bound estimate.

\begin{lemma}\label{density_low_bound_3.1}
Let $(\rho(t,x), u(t,x), S_0(x))$ be a $C^1$ solution of the Cauchy problem
\eqref{lagrangian1q}--\eqref{lagrangian3q} and \eqref{initial1}
defined on the time interval $[0, T)$ for some $T>0$,
with initial data $(\rho_0(x), u_0(x), S_0(x))$ satisfying
Conditions {\rm\ref{initialdata2}}--{\rm\ref{bv}}.
If $1<\gamma<3$,  then, for any $x\in \mathbb{R} $ and $t\in [0, T)$,
there is a positive constant $K_{**}$ depending only on $\gamma$ and $M_U$ such that
\[
v(t,x)\le \Big({v_0}^{\frac{3-\gamma}{4}}(x)+K_{**}({\bar Y}+{\bar Q}) t\Big)^{\frac{4}{3-\gamma}}.
\]
\end{lemma}

\begin{proof} From
\eqref{lagrangian1q}, \eqref{r_s_def}, and Lemma \ref{full_lemma1},
it is clear that
\begin{align*}
v_t=u_x &=\frac12 (r_x+s_x)=\hat{S}^{\frac{3(3-\gamma)}{2(3\gamma-1)}}\,
       a^{-\frac{\gamma+1}{2(\gamma-1)}}(\alpha+\beta)\\
     &\le M_U^{\frac{3(3-\gamma)}{2(3\gamma-1)}}\,
       a^{-\frac{\gamma+1}{2(\gamma-1)}}\big({\bar Y}+{\bar Q}\big).
\end{align*}
Then, using \eqref{z def}, we have
$$
v^{-\frac{\gamma+1}{4}}v_t\le  M_U^{\frac{3(3-\gamma)}{2(3\gamma-1)}}\,
       \big(\TS\frac{2\sqrt{K\gamma}}{\gamma-1}\big)^{-\frac{\gamma+1}{2(\gamma-1)}}
       \big({\bar Y}+{\bar Q}\big),
$$
which implies
$$
v(t,x)\le \big(v_0(x)+ K_{**} ({\bar Y}+{\bar Q}) t\big)^{\frac{4}{3-\gamma}},
$$
where
$$
K_{**}=\frac{3-\gamma}{4}M_U^{\frac{3(3-\gamma)}{2(3\gamma-1)}}\,
       \big(\TS\frac{2\sqrt{K\gamma}}{\gamma-1}\big)^{-\frac{\gamma+1}{2(\gamma-1)}}.
$$
\end{proof}

We now prove Theorem \ref{Thm singularity2}.

\medskip
\noindent
{\it Proof of Theorem {\rm \ref{Thm singularity2}}}.
It suffices to show the finite-time blowup of $\alpha$ and $\beta$
under assumption \eqref{yq-N}.
Without loss of generality, assume that
\beq\label{-N2}
\alpha(0,x^*)<-N \qquad\,\,\mbox{for some $x^*\in \RR$}.
\eeq
We now prove the blowup of $\alpha$ along the forward characteristic starting
from $(0,x^*)$. The proof is similar when $\beta<-N$ somewhere initially.

According to (\ref{-N2}), there exists some $\eps>0$ such that
\beq\label{SS90}
  \alpha(0,x^*) <- (1+\eps)\,N.
\eeq

Denote
\[
x=\Gamma(t; x^*)\qquad\hbox{or equivalently}\quad t=\Omega(x; x^*)
\]
as the forward characteristic starting from $(0,x^*)$.
From Lemma \ref{density_low_bound_3.1}, we have
$$
x=\Gamma(t; x^*)=x^*+\int_0^t c(s,\Gamma(s; x^*))\,\text{d}s
$$
is strictly increasing on $t$, and $\Omega(x; x^*)$ is strictly increasing on $x$.
Then we know that
$\lim_{t\rightarrow\infty}\Gamma(t; x^*)$
must exist or equal to infinity,
due to the monotonicity of $\Gamma(t; x^*)$ for fixed $x^*$.
Accordingly, we divide the rest of the proof into two cases.

\medskip
\paragraph{\bf Case 1}
\, $\lim_{t\rightarrow\infty}\Gamma(t; x^*)=\infty$.
$\,$ For any $t\ge 0$, $\Gamma(t; x^*)$ is well-defined and takes its value in $[x^*,\infty]$.
By \eqref{a0a2n2}--\eqref{a0a2n3}
and the comparison theorem for ODEs,
$$
\alpha(t,\Gamma(t; x^*)) \le \alpha(0,x^*)  < -(1+\eps)\,N.
$$
Therefore, we have
$$
\frac{\alpha^2(t,\Gamma(t; x^*)) }{(1+\eps)^2}> N^2\ge \frac{b_0}{b_2},
$$
which implies
\[
 \pp \alpha(t,\Gamma(t; x^*)) = b_2\Big(\frac{b_0}{b_2}-\alpha^2(t,\Gamma(t; x^*))\Big)
  <-{\TS\frac{\eps(2+\eps)}{(1+\eps)^2}}\, b_2\,\alpha^2(t,\Gamma(t; x^*)).
\]
Then
\[
\frac{\partial \alpha(\Omega(x; x^*), x)}{\partial x}=\frac{b_2}{c} \, \Big(\frac{b_0}{b_2}-\alpha^2\Big)
<-{\frac{\eps(2+\eps)}{(1+\eps)^2}}\,\frac{b_2}{c}\,\alpha^2(\Omega(x; x^*), x).
\]
Integrating the above in $t$ leads to
\[
  \frac{1}{\alpha(\Omega(x; x^*), x)} \ge {\frac{1}{\alpha(0,x^*)} + {\frac{\eps(2+\eps)}{(1+\eps)^2}}
    \int_{x_*}^x K_5 \rho^{\frac{1-3\gamma}{4}}(\alpha(\Omega(z; x^*), z))\;\dd z},
\]
where
\begin{align}\label{a2*}
K_5:= {\TS\frac{(\gamma+1)K_c}{2(\gamma-1)\sqrt{\gamma K}}}\, \hat{S}^{\frac{11-9\gamma}{2(3\gamma-1)}}
        \Big(\frac{2\sqrt{\gamma K}}{\gamma-1}\Big)^{\frac{3-\gamma}{2(\gamma-1)}}.
\end{align}
Using (\ref{m_bounds}), we see that there exist two positive constants $A_1$ and $A_2$ such that
\begin{equation}\label{mbound}
 0<A_1\leq K_5\leq A_2<\infty.
\end{equation}

From Lemma \ref{Thm_upper}, we have
\begin{equation}\label{inversebound2}
\rho^{\frac{1-3\gamma}{4}}(t,x)\geq M^{\frac{1-3\gamma}{4}}_\rho \qquad\,\, \text{for any}\ (t,x)\in [0,T]\times \mathbb{R}.
\end{equation}
Then it follows from (\ref{mbound})--(\ref{inversebound2}) that
\beq{\label{blowupa2}}
\int_{x_*}^\infty \, \big(\frac{b_2}{c}\big)(\Omega(z; x^*), z) \;\dd z
=\infty= \int_{x_*}^\infty K_5 \rho^{\frac{1-3\gamma}{4}}(\Omega(z; x^*), z)\;\dd z.
\eeq
Thus, there must be some point $x_0>x_*$ such that
$$
{\frac{1}{\alpha(0,x^*)} + {\frac{\eps(2+\eps)}{(1+\eps)^2}}
    \int_{x_*}^x K_5 \rho^{\frac{1-3\gamma}{4}}(\alpha(\Omega(z; x^*), z))\;\dd z}=0$$
under assumption \eqref{-N2}.
This implies that $\alpha$ blows up in a finite time.

\smallskip
\medskip
\paragraph{\bf Case 2}
{\it If Case {\rm 1} fails, there must be
\beq{\label{aaa1}}
\lim_{t\rightarrow\infty}\Gamma(t; x^*)=\hbox{L}.
\eeq
}

Along this characteristic $\Gamma(t; x^*)$,
we see from the definition of $N$ that, for any $t\ge 0$ such that $\Gamma(t; x^*)$ is well-defined,
\[
 \pp \alpha(t,\Gamma(t; x^*))  =b_2\Big(\frac{b_0}{b_2}-\alpha^2\Big)<0
\]
and
$$
 \alpha(t,\Gamma(t; x^*)) \le \alpha(0,x_0)  < -(1+\eps)\,N.
$$
Therefore, we have
$$
\frac{\alpha^2(t,\Gamma(t; x^*)) }{(1+\eps)^2}> N^2\ge \frac{b_0}{b_2},
$$
which implies
\[
 \pp \alpha (t,\Gamma(t; x^*)) = b_2\Big(\frac{b_0}{b_2}-\alpha^2(t,\Gamma(t; x^*))\Big)
       <-{\TS\frac{\eps(2+\eps)}{(1+\eps)^2}}\, b_2\,\alpha^2(t,\Gamma(t; x^*)) \,.
\]
Integrating the above in $t$, we have
\[
  \frac{1}{\alpha(t,\Gamma(t; x^*)) } \ge {\frac{1}{\alpha(0,x_0)} + {\frac{\eps(2+\eps)}{(1+\eps)^2}}
    \int_0^t {b_2}(\sigma,\Gamma(\sigma; x^*)) \;\dd\sigma},
\]
where the integral is along the forward characteristic.
To show that $\alpha$ blows up in a finite time, it suffices to show
\beq\label{blowupa2_2}
\int_0^\infty \, {b_2(t,\Gamma(t; x^*)) }\;\dd t=\infty.
\eeq

When $\gamma\geq 3$, from the definition of $b_2$ in \eqref{adefsq}, we have
$$
b_2\ge  {\TS\frac{(\gamma+1)K_c}{2(\gamma-1)}}\,
	M_L^{\frac{3(3-\gamma)}{2(3\gamma-1)}}\,
        E_U^{\frac{3-\gamma}{2(\gamma-1)}},
$$
so that \eqref{blowupa2_2} follows.

When $1<\gamma<3$, by Lemma \ref{density_low_bound_3.1} and
the definition of $b_2$ in \eqref{adefsq}, we have
\[
b_2(t,x) \geq K_c\,{\TS\frac{\gamma+1}{2(\gamma-1)}}\,
	M_L^{\frac{3(3-\gamma)}{2(3\gamma-1)}}  \big(\TS\frac{2\sqrt{\gamma K}}{\gamma-1}\big)^{\frac{3-\gamma}{2(\gamma-1)}}
   \Big(v_0^{\frac{3-\gamma}{4}}(x)+K_6({\bar Y}+{\bar Q})t\Big)^{-1}
\]
with $x=\Gamma(t; x^*)\in[x^*,L)$.
Since $v_0(x)$ with $x=\Gamma(t; x^*)\in[x^*,L)$ is bounded above
by some constant depending on $x^*$ and $L$, but independent of $t$,
\eqref{blowupa2_2} holds.

\smallskip
Therefore,  for any $\gamma>1$,
$\alpha$ blows up in a finite time.
This completes the proof.

\section{Decompression Phenomenon and Global Continuous Solutions for the Nonisentropic Euler Equations}
By Theorem \ref{p_sing_thm}, we know that, for the isentropic flow,
there exists no global $C^1$ solution as long as the initial data contain any compression.
However, this is not the case for the nonisentropic solutions,
which shows much richer structures of the solutions (also see \cite{youngblake1}).

In this section, we present a new phenomenon -- decompression -- of the nonisentropic Euler flow
by constructing a global continuous solution starting from the initial data that contain a weak compression.
The main idea used in constructing the global continuous solution
is to vary the shock-free example
consisting of a contact discontinuity, a rarefaction, and a compression simple wave,
to the continuous solution,
by stretching the contact wave to a smoothly varying entropy profile.
See Fig. \ref{example12}.

\begin{figure}[htp] \centering
\includegraphics[width=.25\textwidth]{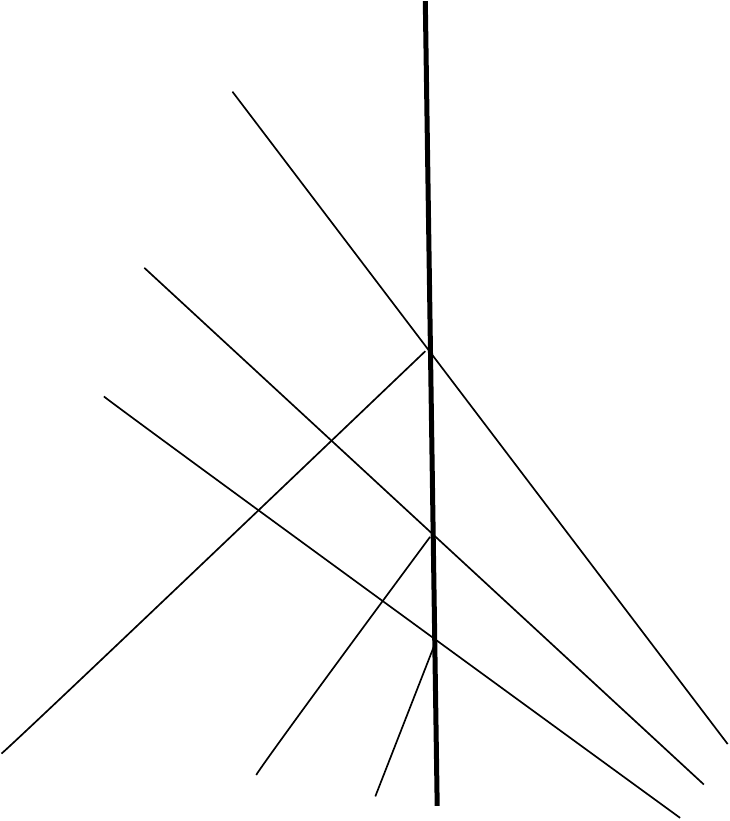}\hspace{20mm}
\includegraphics[width=.3\textwidth]{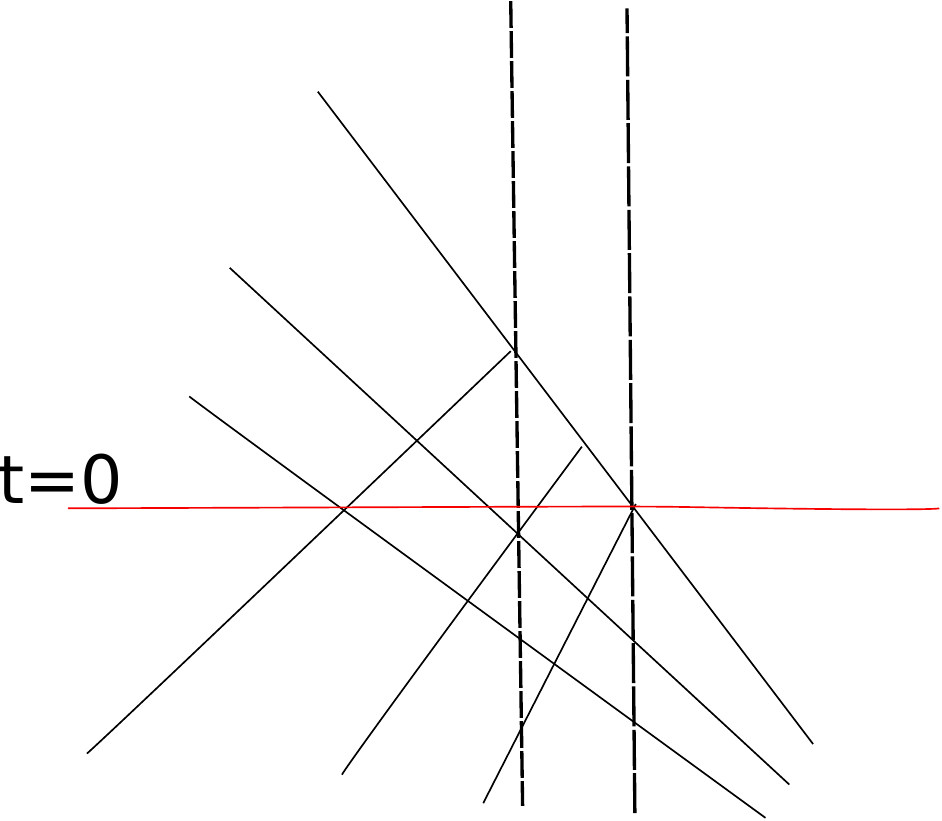}
{
\caption{Vary from a contact wave to a solution with smoothly varying entropy. In the left subfigure, a forward compression
is cancelled by a backward rarefaction at a contact discontinuity. In the right subfigure,
a forward compression is cancelled inside a region with smoothly varying entropy.
Both subfigures are on the $(t,x)$--plane.
\label{example12}}}
\end{figure}
{
Here, the existence of the wave pattern in the left subfigure of Fig. \ref{example12} was proved earlier;
see  \cite{G6,youngblake1}.
The solution is isentropic on each side of the contact discontinuity.
In this section, we prove the existence of the wave pattern in the right subfigure.
Some new difficulties arise compared with  those in  \cite{G6,youngblake1},
since the solution we want to construct is continuous.}

According to Fig. \ref{example12}, it suffices to consider Fig. \ref{example3}.
\begin{figure}[htp] \centering
		\includegraphics[width=.5\textwidth]{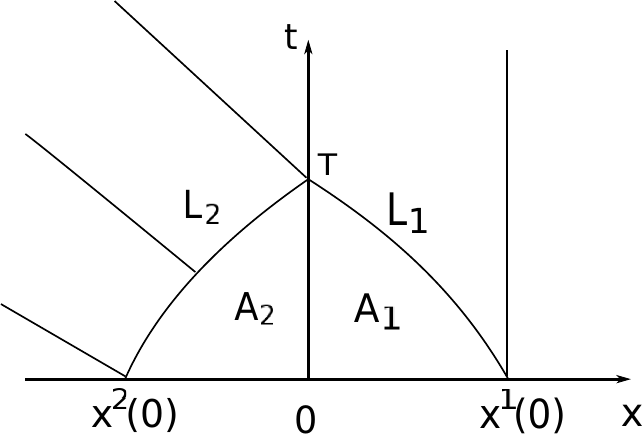}
		\caption{Decompression solutions. This figure is from the right subfigure of Fig. {\rm \ref{example12}},
which shows a backward Goursat problem on region $A(T)=A_1\cup A_2$
with two boundary characteristics $L_1$ and $L_2$. \label{example3}}
	\end{figure}

The key part of construction of the desired global continuous  solution, as shown in  Fig. \ref{example3},
is to solve a backward (in time) Goursat
problem on the angular domain:
\beq\label{AT}
A(T):=A_1\cup A_2=\big\{(t,x)\,:\, x^2(t)\leq x\leq x^1(t),\, t \in [0,T]\big\},
\eeq
with some time $T>0$ and the boundary conditions given on the backward and forward characteristic
boundaries $L_1$ and $L_2$ defined as
\begin{equation}\label{free boundary}
\begin{split}
L_1&:=\Big\{x=x^1(t),\ t\in[0,T]\, :\, \frac{\dd x^1(t)}{\dd t}=-c(t,x^1(t)),\, x^1(T)=0\Big\},\\[6pt]
L_2&:=\Big\{x=x^2(t), \ t\in[0,T]\, :\, \frac{\dd x^2(t)}{\dd t}=c(t,x^2(t)),\, x^2(T)=0\Big\},
\end{split}
\end{equation}
respectively, where $A_1$ and $A_2$ are both closed domains.

The rest of this section is organized as follows: \S 4.1 is devoted to   solving  the  backward (in time) Goursat
problem mentioned above on the angular domain $A(T)$ under some proper boundary conditions.
Then, in \S 4.2, we construct an example of  global continuous solutions
to  a Cauchy problem of  the nonisentropic Euler equations  when the initial data contain a weak compression.

\subsection{The backward Goursat problem}
Now we prove the existence for the backward Goursat problem
on $A(T)=A_1\cup A_2$,
with some given boundary data on $L_1$ and $L_2$ satisfying
\begin{equation}\label{givenproperty}
\xi=0 \,\,\,\, \text{on $L_1$}, \qquad\,\,\,\,  \zeta>0 \,\,\,\, \text{on $L_2$},
\end{equation}
for some time $T>0$  (see \eqref{AT}--\eqref{free boundary}), where $\xi$ and $\zeta$ are
defined in \eqref{def-alpha}--\eqref{def_beta}. In fact, it suffices to prove a local-in-time existence result,
since we can then find some $T>0$ as our starting point,
and the solution still satisfies the equations under any shift on $t$.

Now we prescribe the problem and boundary conditions;
see Fig. \ref{example3}.
First, we denote
\begin{equation}\label{rsnewform}
\tilde{s}=\hat{S}^{-\frac{1}{2\gamma}}s,\qquad \tilde{r}=\hat{S}^{-\frac{1}{2\gamma}}r.
\end{equation}
Then
$$
c=K_c \Big(\frac{\tilde{s}-\tilde{r}}{2}\Big)^{\frac{\gamma+1}{\gamma-1}}\hat{S}^{\frac{1-3\gamma}{2\gamma(\gamma-1)}},
$$
which means that $c=c(\tilde{s}, \tilde{r},\hat{S})$ is a $C^1$ function with respect to $(\tilde{s}, \tilde{r},\hat{S})$.

Using equations $(\ref{fulleuler1})_3$ and (\ref{s_eqn}), we see that  $(\tilde{s}, \tilde{r},\hat{S})$ satisfy
\begin{equation}\label{gusat problem}
\begin{cases}
\partial_+ \tilde{s}=-\frac{1}{2\gamma} \frac{c\hat{S}_x}{\hat{S}} \tilde{r},\\[1.5mm]
\partial_- \tilde{r}=\frac{1}{2\gamma} \frac{c\hat{S}_x}{\hat{S}} \tilde{s},\\[1.5mm]
\hat{S}_t=0.
\end{cases}
\end{equation}
The boundary conditions on $L_1$ and $L_2$ are given as
\begin{equation}\label{boundary condition1}
\begin{split}
L_1&:\ \tilde{s}(t,x^1(t))=h^1(t),\,\,\,\, \hat{S}(t,x^1(t))=f^1(t) \qquad\,\, \mbox{for $0\leq t \leq T$},\\[1mm]
L_2&:\  \tilde{r}(t,x^2(t))=h^2(t),\,\,\,\, \hat{S}(t,x^2(t))=f^2(t)\qquad\,\, \mbox{for $0\leq t \leq T$},
\end{split}
\end{equation}
where the characteristic boundaries $L_1$ and $L_2$ are defined in \eqref{free boundary}, and we assume
\beq\label{c5}
f^1(T)=f^2(T), \qquad c(T,0)
=K_c \Big(\frac{h^1(T)-h^2(T)}{2}\Big)^{\frac{\gamma+1}{\gamma-1}}\big(f^1(T)\big)^{\frac{1-3\gamma}{2\gamma(\gamma-1)}}>0.
\eeq
Then, for sufficiently small $T>0$,
$x^1(t)$ (resp., $x^2(t)$) in $[0,T]$  is a strictly decreasing (resp., increasing) function with respect to $t$.
It is direct to see that the two free boundaries $x^1(t)$ and $x^2(t)$ can be rewritten as
\begin{equation}\label{free boundary*}
\begin{split}
L_1&:\  \frac{\dd \Omega^1(x)}{\dd x}=-\frac{1}{c(\Omega^1(x),x)},\quad \Omega^1(0)=T,\\[0.5mm]
L_2&:\  \frac{\dd \Omega^2(x)}{\dd x}=\frac{1}{c(\Omega^2(x),x)},\quad \ \ \Omega^2(0)=T.
\end{split}
\end{equation}
Thus,  the  boundary conditions (\ref{boundary condition1}) can be given as
\begin{equation}\label{boundary condition2*}
\begin{split}
L_1&\,: \, \tilde{s}(\Omega^1(x),x)=\tilde{h}^1(x),\,\, \hat{S}(\Omega^1(x),x)=\tilde{f}^1(x)
         \qquad\,\, \mbox{for $x\in[0,x^1(0)]$},\\[0.5mm]
L_2&\,: \, \tilde{r}(\Omega^2(x),x)=\tilde{h}^2(x),\,\, \hat{S}(\Omega^2(x),x)=\tilde{f}^2(x)
  \qquad\,\, \mbox{for $x\in[x^2(0),0]$},
\end{split}
\end{equation}
satisfying
the following conditions:

\begin{condition}\label{construction}
For the boundary conditions \eqref{boundary condition2*} above,
\begin{enumerate}
\item[\rm (i)] $\tilde{f}^i(x)$, $i=1,2$, satisfy
\begin{equation}\label{smoothboundaryvalue}
\begin{cases}
  \tilde{f}^1(x)\in  \ C^1([0, x^1(0)]),\,\,\,  C^{-1} \leq  \tilde{f}^1(x) \leq C
       \,\,\,\, &\text{for $x\in [0,x^1(0)]$},\\
  \partial_x\tilde{f}^1(0)=0,\,\,\,
  \partial_x\tilde{f}^1(x)>0  &\text{for $x\in (0,x^1(0))$},\\
  \tilde{f}^2(x)\equiv    \tilde{f}^1(0)   &\text{for $x\in [x^2(0),0]$},
\end{cases}
\end{equation}
for some positive constant $C${\rm ;}

\smallskip
\item[\rm (ii)]
$\tilde{h}^1(x)$ and $\tilde{h}^2(x)$ satisfy
\begin{equation}\label{boundary condition3}
\begin{cases}
\displaystyle
 \tilde{h}^1(x)=\big(\tilde{f}^1(x)\big)^{-\frac{1}{2\gamma}}
 \Big(u_0+K^{\frac{1-\gamma}{2\gamma}}_p\big(\tilde{f}^1(x)\big)^{\frac{1}{\gamma}}p^{\frac{\gamma-1}{2\gamma}}_0\Big),\\[1mm]
 \displaystyle
 \tilde{h}^2(x)=\big(\tilde{f}^2(x)\big)^{-\frac{1}{2\gamma}}h_*(x),
\end{cases}
\end{equation}
where $u_0$ and $p_0>0$ are both constants, and $h_*(x)\in C^1([x^2(0),0])$ satisfies
\begin{equation}\label{increasing entropy}
h_*(0)=u_0- K^{\frac{1-\gamma}{2\gamma}}_p\big(\tilde{f}^2(0)\big)^{\frac{1}{\gamma}}p^{\frac{\gamma-1}{2\gamma}}_0,
\quad\,\,
\partial_x h_*(x)>0 \,\,\,\, \text{for $x\in [x^2(0),0]$}.
\end{equation}
\end{enumerate}
\end{condition}

It is direct to check that there exists functions $\tilde{h}^i$ and $\tilde{f}^i$, $i=1,2$, satisfying  Condition {\rm\ref{construction}}.
Then we have the following local-in-time existence for the backward Goursat problem,
which serves as the key part of the construction of
the desired global continuous solution as shown in  Theorem \ref{mainexistence} later:

\begin{theorem}\label{goursatlocal}
Let Condition {\rm\ref{construction}} hold.
Then there exist a time $T>0$  and a unique $C^1$ solution $(\tilde{s}, \tilde{r},\hat{S})$
of the backward Goursat problem \eqref{gusat problem} with \eqref{free boundary*}--\eqref{boundary condition2*},
or equivalently \eqref{gusat problem} with \eqref{free boundary} and \eqref{boundary condition1}, in $A(T)$ defined in \eqref{AT}.
Furthermore,
\begin{align}
&\xi=0\qquad \hbox{on}\ L_1, \label{c4}\\
&\zeta>0\qquad \hbox{in}\ A(T),  \label{c3}\\
&\xi<0 \qquad \text{in $A(T)\setminus (L_1\cup L_2)$}.\label{c5-a}
\end{align}
\end{theorem}
\smallskip

\begin{proof} Condition {\rm\ref{construction}} implies that $A_1$ is a nonisentropic region and $A_2$ is an isentropic region,
which will be used directly to make sure that property \eqref{givenproperty} $($see \eqref{c4}--\eqref{c3}$)$
holds on $L_1$ and $L_2$, whose details will be given below.
For the convenience of our proof,
we give the boundary conditions in Condition {\rm\ref{construction}}
through the  independent variable $x$, which can also be  given  equivalently through $t$ by relation \eqref{free boundary*}.
The proof is divided into four steps.

\medskip
\textbf{1.} {\it Local existence of the backward Goursat problem}.
In order to solve the backward Goursat problem (\ref{gusat problem}) with \eqref{free boundary} and (\ref{boundary condition1}),
we reformulate it as a fixed boundary problem under some coordinate transformations.

First, we reformulate the backward Goursat problem (\ref{gusat problem}) with \eqref{free boundary} and (\ref{boundary condition1})
by introducing the following new coordinates:
\begin{equation}\label{normalchange}
\tau=T-t,\quad y=x.
\end{equation}
Denote
\begin{equation}\label{rsnewform1}
\tilde{s}_*(\tau,y)=\tilde{s}(T-\tau,y),\quad \tilde{r}_*(v,y)=\tilde{r}(T-\tau,y), \quad \hat{S}(y)=\hat{S}|_{x=y}.
\end{equation}
Then equations (\ref{gusat problem}) can be rewritten as
\begin{equation}\label{gusat problem-refor}
\begin{cases}
\partial_- \tilde{s}_*=\frac{1}{2\gamma} \frac{\tilde{c}\hat{S}_y}{\hat{S}} \tilde{r}_*,\\[1mm]
\partial_+ \tilde{r}_*=-\frac{1}{2\gamma} \frac{\tilde{c}\hat{S}_y}{\hat{S}} \tilde{s}_*,\\[1mm]
\partial_{\tau}\hat{S}=0,
\end{cases}
\end{equation}
where
\begin{equation}\label{free boundary-refor}
\begin{split}
&\tilde{c}(\tau,y)=c(T-\tau,y,\tilde{s}_*, \tilde{r}_*)=K_c \Big(\frac{\tilde{s}_*-\tilde{r}_*}{2}\Big)^{\frac{\gamma+1}{\gamma-1}}m^{\frac{1-3\gamma}{2\gamma(\gamma-1)}},\\
&\partial_-=\partial_\tau-\tilde{c} \partial_y,\quad    \partial_+=\partial_\tau+\tilde{c} \partial_y.
\end{split}
\end{equation}

The new vertex for the reformulated Goursat problem is $\hat{O}=(0,0)$,
and the two free boundaries can be given by
\begin{equation}\label{free boundary-refor_2}
\begin{split}
\hat{L}_1&:\  \frac{\dd y^1(\tau)}{\dd\tau}=\tilde{c}(\tau,y^1(\tau)),\quad y^1(0)=0,\\[1mm]
\hat{L}_2&:\  \frac{\dd y^2(\tau)}{\dd\tau}=-\tilde{c}(\tau,y^2(\tau)),\quad y^2(0)=0.
\end{split}
\end{equation}
Then the time-space domain under consideration is
$$
\tilde{A}(T)=\big\{(\tau,y)\,:\, y^2(\tau)\leq y\leq y^1(\tau), \tau \in [0,T]\big\},
$$
and  the boundary conditions are
\begin{equation}\label{boundary condition-refor}
\begin{split}
\hat{L}_1&:\ \tilde{s}_*(\tau,y^1(\tau))=n^1(\tau),\,\,\, \hat{S}(\tau,y^1(\tau))=l^1(\tau) \qquad\, \text{for $0\leq \tau \leq T$},\\[1mm]
\hat{L}_2&:\  \tilde{r}_*(\tau,y^2(\tau))=n^2(\tau),\,\,\, \hat{S}(\tau,y^2(\tau))=l^2(\tau)\qquad\, \text{for $0\leq \tau \leq T$},
\end{split}
\end{equation}
where $n^i=h^i(T-\tau)$ and $l^i=f^i(T-\tau)$, $i=1,2$, are all $C^1$ functions in $[0,T]$.

\medskip
Next, we reformulate the above normal Goursat problem  (\ref{gusat problem-refor})--(\ref{boundary condition-refor})
into a fixed boundary problem by introducing  the following transformation:
\begin{equation}\label{fixboundary}
\overline{t}=\tau,\quad \overline{x}=\frac{y-y_2(\tau)}{T(\tau)}, \quad T(\tau)=\frac{y_1(\tau)-y_2(\tau)}{\tau}.\end{equation}
The two characteristic free boundaries can be formulated as the following two fixed boundaries:
$$
 \Gamma_1\,:\,\overline{t}=\overline{x}, \qquad\,\,\,  \Gamma_2\,:\,\overline{x}=0.
$$
The angular domain $\tilde{A}(T)$ has also been changed into the following domain:
$$
R(T)=\big\{(\overline{t},\overline{x})\;:\; 0\leq \overline{t}\leq T, \,\, 0\leq \overline{x}\leq \overline{t}\big\}.
$$

Denote
\begin{equation*}
\begin{array}{lll}
&X(\overline{t},\overline{x})=y_2(\overline{t})+\overline{x}T(\overline{t}),\quad &h(\overline{t},\overline{x})=\tilde{s}_*(\overline{t},X(\overline{t},\overline{x})),\\[2mm]
&g(\overline{t},\overline{x})=\tilde{r}_*(\overline{t},X(\overline{t},\overline{x})), \quad &\overline{c}(\overline{t},\overline{x})=\tilde{c}(\overline{t},X(\overline{t},\overline{x}),h,g),\\[2mm]
& \overline{\hat{S}}(\overline{t},\overline{x})=\hat{S}(X(\overline{t},\overline{x})),\quad  &\overline{\hat{S}_y}(\overline{t},\overline{x})=\hat{S}_y(X(\overline{t},\overline{x})),\\[2mm]
&\mu_1(\overline{t},\overline{x},h,g)=\frac{\overline{c}}{2\gamma}\,\frac{ \overline{\hat{S}_y}}{\overline{\hat{S}}} g,\quad
   &\mu_2(\overline{t},\overline{x},h,g)=-\frac{\overline{c}}{2\gamma}\,\frac{ \overline{\hat{S}_y}}{\overline{\hat{S}}} h.
\end{array}
\end{equation*}
From (\ref{gusat problem-refor}), we obtain the new system:
\begin{equation}\label{fixedboundaryproblem}
\begin{cases}
\partial_{\overline{t}} h-(\overline{c}+\partial_{\overline{t}}X)\frac{1}{T(\overline{t})}\partial_{\overline{x}}h=\mu_1,\\[2mm]
\partial_{\overline{t}} g+(\overline{c}-\partial_{\overline{t}}X)\frac{1}{T(\overline{t})}\partial_{\overline{x}}g=\mu_2,
\end{cases}
\end{equation}
and the boundary conditions on $\Gamma_i$, $i=1,2$:
\begin{equation}\label{fixedboundarycondition}
\begin{cases}
\Gamma_1(\overline{t}): \quad h(\Gamma_1(\overline{t}))=n^1(\overline{t})\in C^1([0,T]),\\[1.5mm]
\Gamma_2(\overline{t}):\quad g(\Gamma_2(\overline{t}))=n^2(\overline{t})\in C^1([0,T]),
\end{cases}
\end{equation}
where $n^i(\overline{t})=h^i(T-\overline{t})$, $i=1,2$.

We now prove the desired local well-posedness of the boundary value problem (\ref{fixedboundaryproblem})--(\ref{fixedboundarycondition}),
which actually implies the desired local existence in  Theorem \ref{goursatlocal}.

In fact, the local-in-time well-posedness of the  boundary value problem in some angular domain has been
established in  Theorem \ref{thA} in Appendix A.
Now we only need to show that the coefficients of the equations
and  the boundary conditions  of  problem (\ref{fixedboundaryproblem})--(\ref{fixedboundarycondition})
satisfy the corresponding requirements in Theorem \ref{thA}:

\smallskip
(i) For the boundary conditions,
since $h^i$ and $f^i$ are all $C^1$ functions in $[0,T]$ for $i=1,2$,
then
$$
n^i(\overline{t})=h^i(T-\overline{t}), \quad l^i(\overline{t})=f^i(T-\overline{t}), \qquad\,\, i=1,2,
$$
are also all $C^1$ functions in $[0,T]$.

\smallskip
(ii) For the coefficients, it is clear that
\begin{equation*}
\begin{split}
&\overline{\lambda}_1(\overline{t},\overline{x},h,g)=-(\overline{c}+\partial_{\overline{t}}X)\frac{1}{T(\overline{t})},\quad \overline{\lambda}_2(\overline{t},\overline{x},h,g)=(\overline{c}-\partial_{\overline{t}}X)\frac{1}{T(\overline{t})},\\[1mm]
&\mu_1(\overline{t},\overline{x},h,g)=\frac{\overline{c}}{2\gamma}\,\frac{ \overline{\hat{S}_y}}{\overline{\hat{S}}} g,
\quad  \mu_2(\overline{t},\overline{x},h,g)=-\frac{\overline{c}}{2\gamma}\,\frac{ \overline{\hat{S}_y}}{\overline{\hat{S}}} h,\\[1mm]
&(\overline{\lambda}_i)_x,\, (\mu_i)_x,\, (\overline{\lambda}_i)_h,\, (\overline{\lambda}_i)_g,\,
(\mu_i)_h,\, (\mu_i)_g, \qquad\,\, i=1,2,
\end{split}
\end{equation*}
are
all continuous  with respect to all the variables $(\overline{t},\overline{x},h,g)$ in their corresponding domains.

\smallskip

(iii)  It follows from \eqref{free boundary-refor_2}, \eqref{fixboundary}, and a direct calculation that
$$
\overline{\lambda}_1(\overline{t},\overline{x},h,g)|_{\Gamma_2}=0, \qquad \overline{\lambda}_2(\overline{t},\overline{x},h,g)|_{\Gamma_1}=1.
$$

Then assumptions (i)--(iii) of Appendix A.2 hold automatically.
Thus,  they satisfy all the requirements of Theorem \ref{thA}.

Based on the above analysis,  Theorem \ref{thA} implies that
there exists some small $\delta>0$  such that  the boundary value problem (\ref{fixedboundaryproblem})--(\ref{fixedboundarycondition})
has a unique $C^1$ solution $(h,g)$ in the angular domain
$$
R(\delta)=\big\{(\overline{t},\overline{x})\,:\, 0\leq \overline{t}\leq \delta, \,\, 0\leq \overline{x}\leq \overline{t}\big\}.
$$
This implies  the desired local well-posedness of our backward Goursat problem in Theorem \ref{goursatlocal}.

\medskip
\textbf{2.} {\it Verification of \eqref{c4}}. Following the above step,
there exists  a unique $C^1$ solution $(\tilde{s}, \tilde{r},\hat{S})$ in $A(T)$ of the backward Goursat problem (\ref{gusat problem})
with \eqref{free boundary} and (\ref{boundary condition1}), when $T$ is small enough.

First, it  follows from Condition \ref{construction} that, at the vertex point
$
O=(T,0)
$,
\begin{equation}\label{vertex point}
u(O)=u_0, \quad p(O)=p_0.
\end{equation}

Next, from (\ref{rsnewform}), $(\ref{gusat problem})_3$, and  (\ref{boundary condition3}), we have
$$
s=u_0+K^{\frac{1-\gamma}{2\gamma}}_p\big(\tilde{f}^1(x)\big)^{\frac{1}{\gamma}}p^{\frac{\gamma-1}{2\gamma}}_0
\qquad \text{on $L_1$}.
$$
Then it follows from a direct calculation that, in $L_1$,
\begin{align}
\partial_-s=&\, s_t-cs_x=-c \frac{\dd s}{\dd x}
= -\frac{c}{\gamma}\, K^{\frac{1-\gamma}{2\gamma}}_p\big(\tilde{f}^1(x)\big)^{\frac{1}{\gamma}-1}\partial_x \tilde{f}^1(x)\, p^{\frac{\gamma-1}{2\gamma}}_0 \label{alongboundary}  \\
=&-\frac{c}{\gamma}\, K^{\frac{1-\gamma}{2\gamma}}_p \hat{S}^{\frac{1}{\gamma}-1}\hat{S}_x\, p^{\frac{\gamma-1}{2\gamma}}_0.\nn
\end{align}
On the other hand, using  Lemma {\rm \ref{riclemma}}, we have
\begin{align}
\partial_-s=&\, s_t-cs_x=-c\xi-c\big(\xi+\frac{1}{\gamma}\hat{S}_xa\big)= -2c\xi -\frac{c}{\gamma}\hat{S}_xa \label{alongboundary1} \\
=& -2c\xi -\frac{1}{\gamma} K^{   \frac{1-\gamma}{2\gamma}}_p c \hat{S}_x p^{\frac{\gamma-1}{2\gamma}}\hat{S}^{\frac{1}{\gamma}-1},\nn
\end{align}
which, together with Lemma {\rm \ref{riclemma}} and  (\ref{alongboundary}), implies
\begin{equation}\label{relation}
\partial_{-}(p-p_0)=\frac{c^2}{2\gamma}K^{\frac{1-\gamma}{2\gamma}}_p \hat{S}^{\frac{1}{\gamma}-1} \hat{S}_x
\Big( p^{\frac{\gamma-1}{2\gamma}}-  p^{\frac{\gamma-1}{2\gamma}}_0\Big).
\end{equation}

From the Gronwall inequality and $p-p_0=0$ at $O=(T,0)$, we have
$$
(p, u)=(p_0, u_0)\qquad\,\, \text{on $L_1$},
$$
which, together with Lemma {\rm \ref{riclemma}},  implies \eqref{c4}.

\medskip
\textbf{3.} {\it Verification of \eqref{c3}}.  In the rest of the proof, we denote $C\geq 1$ as a generic constant depending only on
the boundary conditions on $L_1$ and $L_2$,  and $T$, which may be different at each occurrence.

First, we need   the following lower and upper bounds on $(c,a)$:
\begin{equation}\label{density-lower}
C^{-1}\leq a(t,x), c (t,x)\leq C \qquad\,\, \text{for $(t,x)\in A(T)$},
\end{equation}
if $T$ is sufficiently small.

In fact, from the local existence obtained  in Step 1, we know
\begin{equation}\label{localbound}
\|(s, r,\hat{S})\|_{C^1(A(T))}\leq C,
\end{equation}
which, along with the definitions of $(s,r)$ and
equation $(\ref{gusat problem})_3$,
implies
\begin{equation}\label{localbound-d}
\|(c\xi, c\zeta)\|_{L^\infty(A(T))}+
\|(a, c)\|_{C^1(A(T))} \leq  C
\end{equation}
for some constant $C>0$.
Furthermore, the boundary conditions in (\ref{boundary condition3}) on $L_1$ and $L_2$ imply
\begin{equation}\label{densityboundinl6}
a (O)\geq C^{-1}
\end{equation}
for $O=(T,0).$
On curve $L_2$,
it follows from (\ref{localbound-d}) that
\begin{equation}\label{densityboundinl6-1}
\|\partial_+a\|_{L^\infty(L^2)}\leq C,
\end{equation}
which, together with (\ref{densityboundinl6}), implies
\begin{equation}\label{densityboundinl6-2}
a (t,x_2(t))\geq C^{-1} \qquad \text{for $0\leq t \leq T$},
\end{equation}
if $T$ is sufficiently small.
Similarly, along the backward characteristic, we have
\begin{equation}\label{densityboundinl6-3}
a (t,x), \ c (t,x) \geq C^{-1} \qquad\,\, \text{for $(t,x)\in A(T)$}.
\end{equation}
Then we obtain the desired lower and upper bounds in \eqref{density-lower}.

Next, it follows from  equations (\ref{gusat problem})
and the boundary conditions (\ref{boundary condition3}) that
\begin{equation}\label{timeevolution}
\partial_+ \tilde{s}=0, \quad \partial_- \tilde{r}=0 \qquad\,\,\, \text{on $L_2$}.
\end{equation}
Then
\begin{equation}\label{alongboundary-r}
s(t,x^2(t))=s(O), \quad r_t -cr_x=0  \qquad\,\,\, \text{on $L_2$},
\end{equation}
which, together with Lemmas {\rm \ref{riclemma}} and {\rm \eqref{density-lower}}, and $\partial_+ r>0$ on $L_2$, implies
\begin{align}
\zeta=&\,\frac{r_t}{c}=\frac{1}{2c}(r_t+cr_x)=\frac{1}{2}\partial_x h_*(x)\label{relation_2} \\
\geq&\, \frac{1}{2}\min_{x\in [x^2(0),0]}\partial_x h_*(x)>  0 \qquad\,\, \text{on $L_2$}.\nn
\end{align}

Finally,  let $x(t; x_0,T^*)$ be one backward characteristic starting from $(0,x_0)$ and
connecting to line $L_2$ by $(T^*,x(T^*))$.
From (\ref{frem1}), we have
\begin{equation}\label{weightedestimates-2}
\zeta(\sigma,x(\sigma; x_0,T^*))=\zeta(T^*,x(T^*))
+\int_\sigma^{T^*} k_1\big(k_2
(\xi+3\zeta)-\xi\zeta+\zeta^2\big)\,\text{d}\sigma,
\end{equation}
which, together with \eqref{density-lower}, (\ref{localbound-d}), and \eqref{relation_2},
implies
\begin{equation}\label{postive}
\zeta>C^{-1} \qquad\,\, \text{in $A(T)$}
\end{equation}
for some time $T>0$ small enough.

\medskip
\textbf{4.} {\it Verification of \eqref{c5-a}}.
Using Condition \ref{construction}, (\ref{k def}), and equation $(\ref{gusat problem})_3$,
we have
\begin{equation}\label{k1k2}
\begin{split}
& k_1>0\,\,\,\,\,\text{in $A(T)$},\,\,\qquad\,\, \hat{S}_x, k_2 =0 \,\,\,\,\,\text{in $A_2 \cup \{(0,x^1(0))\}$},\\
&\hat{S}_x,k_2 >0 \,\,\,\,\,\text{in $A_1\setminus (\{x=0\}\cup \{(0,x^1(0))\})$}.
\end{split}
\end{equation}
From \eqref{c4} and \eqref{postive}, we have
$$
\xi=0 \,\,\,  \text{on $L_1$}, \qquad\, \zeta>C^{-1} \,\,\, \text{in $A(T)$}.
$$
Then
\begin{equation}\label{partialalpha}
\partial_+\xi=k_1k_2\zeta>0 \qquad\,\,\text{on $L_1\setminus \{(0,x^1(0)),O\}$}
\end{equation}
with $O=(T,0)$, where we have used (\ref{frem1}) and (\ref{k1k2}).

Furthermore, since $\zeta>0$ and $\hat{S}_x\geq0$ in $A(T)$,
by  (\ref{frem1}), $\xi$ does not change from negative to positive in the reverse time
direction along any forward characteristic in $A_1$.
Thus, we see that $\xi<0$ in the interior of $A_1$ and $\xi=0$ on $L_1$.
Recall that $\xi$ does not change sign in the isentropic region $A_2$.
Then
$$
\xi<0 \,\,\,\,\mbox{in $A(T)\setminus (L_1\cup L_2)$},\qquad\,\,
\xi=0 \,\,\,\,\mbox{on $L_1\cup L_2$}.
$$
This completes the proof.
\end{proof}

\begin{remark}\label{rem_RC}
Note that $\xi<0$  on both the isentropic and nonisentropic regions $A_2$ and $A_1${\rm ;}
that is, the forward waves on $A_2$ are of compression in the sense of Definition {\rm\ref{def1}}.
\end{remark}

\subsection{Nontrivial global continuous solution with a weak compression}

Another objective of this paper is to construct an example of  global continuous solutions
for the nonisentropic Euler equations when the initial data contain a weak compression.
We now prove the main theorem of this section; see Fig. \ref{example3}.

\begin{theorem}\label{mainexistence}
Let the initial data $(v_0, u_0, S_0)\in C(\mathbb{R})$ satisfy the following{\rm :}

\smallskip
\begin{itemize}
\item  In interval $[x^2(0), x^1(0)]$,  the vector function $(v_0, u_0, S_0)$
is described as a part of the $C^1$ solution of the  backward Goursat problem in  $A(T)$ of Theorem {\rm\ref{goursatlocal}}{\rm ;}

\smallskip
\item In intervals  $(-\infty, x^2(0)]$ and $[x^1(0),\infty)$,
\beq\label{c2}
\begin{split}
&(u_0,\rho_0,S_0)=(u,\rho,S)(0, x^2(0))\qquad\hbox{when $x\in(-\infty,x^2(0)]$},\\
&(u_0,\rho_0,S_0)=(u,\rho,S)(0, x^1(0))\qquad\hbox{when $x\in[x^1(0), \infty)$}.
\end{split}
\eeq

\end{itemize}
Then there exists a unique nonisentropic global continuous solution $(v, u, S)$
for the Cauchy problem \eqref{lagrangian1q}--\eqref{lagrangian3q} with \eqref{initial1}.
In addition, $(v, u, S)$ satisfies the following{\rm :}
\begin{itemize}

\item[\rm (i)]  $(v, u, S)$ is $C^1$, except on the forward characteristic passing through $(0,x^2(0))$,
  the backward characteristic passing through  $(0,x^1(0))$, and the characteristic
  $x=x^1(0)$  {\rm (}i.e., only Lipschitz continuous{\rm )}{\rm ;}

\smallskip

\item [\rm (ii)]  $S_x(t,x)>0$ for $(t,x)\in [0,\infty)\times (0,x^1(0)]$, and
  $S_x(t,x)=0$  in the region outside of $ [0,\infty)\times (0,x^1(0)]${\rm ;}

\smallskip
\item [\rm (iii)] There exists some finite $T>0$ such that the solution contains a compression
in some region $A_2\subset [0,\infty)\times [0,T]$
in the sense of  Definition {\rm\ref{def1}},
and also $\xi<0$ in some region $A_1\subset [0,\infty)\times [0,T]${\rm ;}

\smallskip
\item [\rm (iv)] $\xi\geq 0$ and $\zeta\geq 0$ in the region outside of $A(T)=A_1\cup A_2$ in the $(t,x)$--plane.
\end{itemize}
\end{theorem}

\begin{proof}
For the region out of $A(T)$,
there is an initial-boundary value problem.
We now set the initial and boundary conditions.

\begin{condition}\label{ass_last}
The two boundaries $L_1$ and $L_2$, and the boundary data on $L^1$ and $L^2$ are
prescribed as in Theorem {\rm\ref{goursatlocal}}. On $x\in(-\infty,x^2(0)]\cup[x^1(0), \infty)$, the initial data satisfy \eqref{c2}.
\end{condition}

Now we prove the global existence on $\{(t,x)\in \mathbb{R}^+\times\mathbb{R}^-\}\setminus A_2$.
First, in this region, the entropy is a constant,
by Conditions \ref{construction} and \ref{ass_last},
and the fact that the entropy is always stationary by using
\eqref{s con} and the boundary condition $\eqref{boundary condition2*}_2$ and $\eqref{smoothboundaryvalue}_3$.
By \eqref{c3}, there is a backward rarefaction on $\{(t,x)\in \mathbb{R}^+\times\mathbb{R}^-\}\setminus A_2$.
Thus, in this region, there exists a unique global isentropic continuous solution
by the $C^1$ global existence theory in \cite{Lidaqian}.
More precisely, the classical theory works for the region below the backward characteristic $x=x^1(t),\ t\in[T, \infty]$,
since there exists a $C^1$ bound on the rarefaction wave.
Between $x=x^1(t)$ for $t\in[T, \infty]$ and the positive $t$ axis,
the unknowns take constant values following the initial data at point $(T,0)$.
Across the characteristics passing through $(0,x^2(0))$ and $(0,x^1(0))$,
the solution is Lipschitz continuous but not smooth.

Then we show the global existence on $\{(t,x)\in \mathbb{R}^+\times\mathbb{R}^+\}\setminus A_1$.
Using \eqref{forward} and \eqref{c4}, we see that, on $L_1$ for $t\in[0,T]$,
\[
u(t,x)=u(0,x^1(0)),\quad p(t,x)=p(0,x^1(0)).
\]
Furthermore, by \eqref{c2},
\[
u(0,x)=u(0,x^1(0)),\quad p(0,x)=p(0,x^1(0))\qquad\,\,\, \hbox{for $x\in [x^1(0),\infty)$}.
\]
Then
\[
u(t,x)=u(0,x^1(0)),\quad p(t,x)=p(0,x^1(0))\qquad\,\, \hbox{for any $(t,x)\in \mathbb{R}^+\times\mathbb{R}^+\setminus A_1$},
\]
together with the stationary entropy profile prescribed in Condition \ref{ass_last},
give the unique piecewise smooth solution in $\mathbb{R}^+\times\mathbb{R}^+\setminus A_1$.
In fact, this can be seen by Lemma {\rm \ref{riclemma}} and the fact that $\xi=\zeta=0$
is an equilibrium for the equations in Lemma \ref{riclemma}.

Combining these with Theorem \ref{goursatlocal}, we conclude the proof of Theorem \ref{mainexistence}.
\end{proof}

\begin{remark}\label{rem_ns}
Between the characteristics $L_1$ and $L_2$ where the nonsmoothness occurs,
there is a backward rarefaction wave. Beyond region $A(T)$, this wave is particularly an isentropic rarefaction simple wave.
Note that any centered rarefaction simple wave, solved from the Riemann problems,
also has two nonsmooth boundaries \cite{smoller}.
Therefore, the nonsmooth boundaries in our example are quite natural
in the solutions of hyperbolic conservation laws.
\end{remark}

The global continuous solution obtained in Theorem \ref{mainexistence} shows a new phenomenon,
different from the isentropic case, that
the initial weak compression  (in some isentropic region)  might be cancelled during the wave
interaction (after crossing some nonisentropic region) owing   to the effect of the variation
of the entropy.
See Remark \ref{rem_ns} also for an explanation on the nature of such a weak discontinuity
shown in (i) of the above theorem.
We conjecture that the weak discontinuity could be removed by some further development of
our construction method.

To our knowledge, this is the first example of  global continuous solutions with large initial data that contain a compression
in the sense of Definition \ref{def1}.
In fact, according to Definition \ref{def1} of compression/rarefaction characters for isentropic flow and the ones  in  \cite{G3, G5, G6}
for the nonisentropic flow, the example in Theorem \ref{mainexistence} includes both isentropic
and nonisentropic compression waves (Remark \ref{rem_RC}),
and  the example in \cite{linliuyang} includes only rarefaction waves ($\xi\geq 0$ and $\zeta \geq 0$).
In addition, we provide a new method for constructing concrete nonisentropic
continuous solutions with large initial data that contain weak compressions.

\section{De-rarefaction Phenomenon and Shock Formation for the Nonisentropic Euler Equations \label{sec_derare}}

In this section,  we present a new phenomenon of de-rarefaction for the
nonisentropic Euler flows by constructing one  continuous solution with nonisentropic rarefactive initial data
which forms a shock wave in a finite time.
More precisely, the initial data include a pair of forward and backward isentropic rarefaction simple waves
and a locally stationary solution with varying entropy.

There is no ambiguity on the definition of isentropic rarefaction simple waves,
which mean, in the isentropic domain, one of the Riemann invariants
is constant while the other increases on $x$.
Notice that
the stationary solution with varying entropy
includes no rarefaction and compression phenomena.
In fact, such a solution can stay stationary forever if it is not perturbed by other waves.

\begin{figure}[htp] \centering
\includegraphics[width=.42\textwidth]{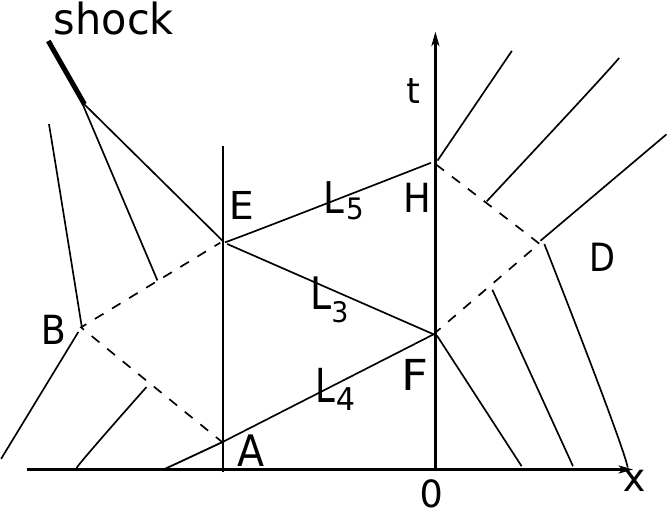}
\caption{De-rarefaction solutions. The entropy is constant to the left of $AE$ and to the right of $HF$, respectively.
The entropy is increasing in $x$ between  $AE$ and $HF$. Lines $L_i$, $i=3,4,5$, are the boundary characteristics
of the Goursat problems in FEA and EFH.
\label{Fig:example4}}
\end{figure}

The rest of this section is organized as follows:
\S 5.1 is devoted to  solving the left oriented Goursat problem in FEA,
and \S 5.2  is devoted to solving the right oriented Goursat problem in EFH, as shown in Fig. \ref{Fig:example4}.
Based on the existence theories of these two Goursat
problems, in \S 5.3, we succeed in constructing the first continuous solution
with nonisentropic rarefactive initial data of a Cauchy problem, so that the solution forms a shock wave in finite time.

\subsection{Left oriented Goursat problem in FEA}	
Denote
$$
F=(T,0), \quad  E=(T_E, X_E)=(T_E, X_A),\quad A=(T_A,X_A),
$$
where $T, T_E, T_A>0$.

We consider the left oriented Goursat problem  of system
 \eqref{gusat problem} in the angular domain $FEA$,
where the backward and forward characteristic
boundaries $L_3$ and $L_4$ are defined as
\begin{equation}\label{free boundary-1}
\begin{split}
L_3 (EF) &:=\Big\{x=x^3(t),\ t\in[T,T_E]\,:\, \frac{\dd x^3(t)}{\dd t}=-c(t,x^3(t)),\ x^3(T)=0\Big\},\\[3pt]
L_4 (AF) &:=\Big\{x=x^4(t), \ t\in[T_A,T]\, :\, \frac{\dd x^4(t)}{\dd t}=c(t,x^4(t)),\ x^4(T)=0\Big\}.
\end{split}
\end{equation}

The boundary conditions on $L_3$ and $L_4$ are given as
\begin{equation}\label{boundary condition-1}
\begin{split}
L_3 &:\ \tilde{s}(t,x^3(t))=h^3(t),\,\,\, \hat{S}(t,x^3(t))=f^3(t) \qquad\, \mbox{for $T\leq t \leq T_E$},\\[0.5mm]
L_4 &:\  \tilde{r}(t,x^4(t))=h^4(t),\,\,\, \hat{S}(t,x^4(t))=f^3(t)\qquad\, \mbox{for $0\leq t \leq T$},
\end{split}
\end{equation}
with
\beq\label{c5-1}
c(T,0)
=K_c \Big(\frac{h^3(T)-h^4(T)}{2}\Big)^{\frac{\gamma+1}{\gamma-1}}\big(f^3(T)\big)^{\frac{1-3\gamma}{2\gamma(\gamma-1)}}>0.
\eeq
Then, for sufficiently small $\epsilon>0$,
$x^3(t)$ (resp., $x^4(t)$) in $[T-\epsilon,T+\epsilon]$  is a strictly decreasing (resp., increasing) function with respect to $t$.
It is easy to see that the two free boundaries $x^3(t)$ and $x^4(t)$ can be rewritten as
\begin{equation}\label{free boundary*-1}
\begin{split}
L_3&:\  \frac{\dd \Omega^3(x)}{\dd x}=-\frac{1}{c(\Omega^3(x),x)},\quad \Omega^3(X_F)=T,\\[1mm]
L_4&:\  \frac{\dd \Omega^4(x)}{\dd x}=\frac{1}{c(\Omega^4(x),x)},\quad \ \ \Omega^4(X_F)=T.
\end{split}
\end{equation}
Thus,  the  boundary conditions (\ref{boundary condition-1}) can be given as
\begin{equation}\label{boundary condition2*-1}
\begin{split}
L_3&:\, \tilde{s}(\Omega^3(x),x)=\tilde{h}^3(x),\,\, \hat{S}(\Omega^3(x),x)=\tilde{f}^3(x) \qquad\, \mbox{for $x\in[X_A,0]$},\\[1mm]
L_4&: \,  \tilde{r}(\Omega^4(x),x)=\tilde{h}^4(x),\,\, \hat{S}(\Omega^4(x),x)=\tilde{f}^3(x)\qquad\, \mbox{for $x\in[X_A,0]$},
\end{split}
\end{equation}
which satisfy
the following condition:

\begin{condition}\label{construction-1}
$\tilde{f}^3(x)$, $\tilde{h}^3(x)$, and $\tilde{h}^4(x)$ satisfy the following properties{\rm :}

\smallskip
\begin{enumerate}
\item[\rm (i)] $\tilde{f}^3(x)$ satisfies that, for $x\in [X_A,0]$,
\begin{equation}\label{smoothboundaryvalue-1}
 \tilde{f}^3(x)\in  \ C^1([X_A,0]), \quad C^{-1} \leq  \tilde{f}^3(x) \leq C,\quad    \partial_x\tilde{f}^3(x)>0,
\end{equation}
where $C$ is  some positive constant{\rm ;}

\smallskip
\item[\rm (ii)]
$\tilde{h}^3(x)$ and $\tilde{h}^4(x)$ satisfy that, for $x\in [X_A,0]$,
\begin{equation}\label{boundary condition3-1}
\begin{cases}
\displaystyle
 \tilde{h}^3(x)=\big(\tilde{f}^3(x)\big)^{-\frac{1}{2\gamma}}h^*(x),\\[0.5mm]
\displaystyle
 \tilde{h}^4(x)=\big(\tilde{f}^3(x)\big)^{-\frac{1}{2\gamma}}\Big(u_0-K^{\frac{1-\gamma}{2\gamma}}_p
 \big(\tilde{f}^3(x)\big)^{\frac{1}{\gamma}}p^{\frac{\gamma-1}{2\gamma}}_0\Big),
\end{cases}
\end{equation}
where $u_0$ and $p_0>0$ are both constants, and $h^*(x)$ satisfies
\begin{equation}\label{increasing entropy-1}
\begin{split}
&h^*(x)\in C^1([X_A,0]),\quad h^*(0)
  =u_0+K^{\frac{1-\gamma}{2\gamma}}_p\big(\tilde{f}^3(0)\big)^{\frac{1}{\gamma}}p^{\frac{\gamma-1}{2\gamma}}_0,\\
&\min_{x\in [X_A,0]}\partial_x h^*(x)-\frac{2}{\gamma}
K^{\frac{1-\gamma}{2\gamma}}_p\big(\tilde{f}^3(0)\big)^{\frac{1}{\gamma}}p^{\frac{\gamma-1}{2\gamma}}_0\max_{x\in[X_A,0]} \partial_x\tilde{f}^3(x) >0.
 \end{split}
\end{equation}
\end{enumerate}
\end{condition}
Then we have the following local-in-time existence for the above Goursat problem:

\begin{theorem}\label{goursatlocal-1}
Let Condition {\rm\ref{construction-1}} hold.
Then there exist a constant $x_*<0$  and a unique $C^1$ solution $(\tilde{s}, \tilde{r},\hat{S})$
of the left oriented Goursat problem \eqref{gusat problem} with \eqref{free boundary*-1}--\eqref{boundary condition2*-1},
or equivalently \eqref{gusat problem} with \eqref{free boundary-1}--\eqref{boundary condition-1}, in $Q(x_*)$ defined by
\beq\label{QQ}
Q(x_*)=\big\{(t,x)\,:\, \Omega^4(x)\leq t\leq \Omega^3(x), \,  x \in [x_*,0]\big\}.
\eeq
Furthermore,
\begin{align}
&\zeta=0\qquad \hbox{on}\ L_4,  \label{c4-1}\\
&\xi>0\qquad \hbox{in}\ Q(x_*), \label{c3-1}\\
&\zeta<0 \qquad \text{in $Q(x_*)\setminus L_4$}.\label{c5-1-a}
\end{align}
Moreover, the length of the life span $|x_*|$ depends only on the $C^1$--norm
of  $\tilde{f}^3(x)$, $\tilde{h}^3(x)$, and $\tilde{h}^4(x)$.
\end{theorem}

\begin{proof}
We divide the proof into four steps.

\medskip
\textbf{1}. {\it Local existence of the left oriented Goursat problem in FEA}.
In order to solve the left oriented Goursat problem (\ref{gusat problem}) with \eqref{free boundary-1}--(\ref{boundary condition-1}),
we just need to reformulate it as a normal  Goursat problem under some coordinate transformations.

First, we  introduce the following new coordinates:
\begin{equation}\label{normalchange-1}
(\tau, y)=(t-T, -x).
\end{equation}
Denote
\begin{equation}\label{rsnewform1-1}
\tilde{s}_*(\tau,y)=\tilde{s}(\tau+T,-y),\quad \tilde{r}_*(\tau,y)=\tilde{r}(\tau+T,-y), \quad \hat{S}_*(y)=\hat{S}(-y).
\end{equation}
Then equations (\ref{gusat problem}) can be rewritten as
\begin{equation}\label{gusat problem-refor-1}
\begin{cases}
(\tilde{s}_*)_y-\frac{1}{\tilde{c}}(\tilde{s}_*)_\tau=-\frac{1}{2\gamma} \frac{(\hat{S}_*)_y}{\hat{S}_*} \tilde{r}_*,\\[1mm]
(\tilde{r}_*)_y+\frac{1}{\tilde{c}}(\tilde{r}_*)_\tau=-\frac{1}{2\gamma} \frac{(\hat{S}_*)_y}{\hat{S}_*} \tilde{s}_*,\\[1mm]
(\hat{S}_*)_\tau=0,
\end{cases}
\end{equation}
where
\begin{equation}\label{free boundary-refor-1}
\tilde{c}(\tau,y)=c(\tau+T,-y,\tilde{s}_*, \tilde{r}_*)
=K_c \big(\frac{\tilde{s}_*-\tilde{r}_*}{2}\big)^{\frac{\gamma+1}{\gamma-1}}(\hat{S}_*)^{\frac{1-3\gamma}{2\gamma(\gamma-1)}}.
\end{equation}

The vertex for the reformulated Goursat problem is still $F$,
and the two free boundaries can be given by
\begin{equation}\label{free boundary-refor_2-1}
\begin{split}
\hat{L}_3&:\  \frac{\dd \tau^3(y)}{\dd y}=\frac{1}{\tilde{c}(\tau^3(y),y)},\,\,\, \tau^3(0)=0,\\[1mm]
\hat{L}_4&:\  \frac{\dd \tau^4(y)}{\dd y}=-\frac{1}{\tilde{c}(\tau^4(y),y)},\,\,\, \tau^4(0)=0.
\end{split}
\end{equation}
Thus, the time-space domain under consideration is
$$
\tilde{A}(-X_A)=\big\{(\tau,y)\,:\, \tau^4(y)\leq \tau \leq \tau^3(y),\, y \in [0, -X_A]\big\},
$$
and  the boundary conditions are
\begin{equation}\label{boundary condition-refor-1}
\begin{split}
\hat{L}_3&:\ \tilde{s}_*(\tau^3(y),y)=n^3(y),\,\,\, \hat{S}_*(\tau^3(y),y)=l^3(y) \qquad\, \text{for $0\leq y \leq -X_A$},\\[1mm]
\hat{L}_4&:\  \tilde{r}_*(\tau^4(y),y)=n^4(y),\,\,\, \hat{S}_*(\tau^4(y),y)=l^3(y)\qquad\, \text{for $0\leq y \leq -X_A$},
\end{split}
\end{equation}
where $n^3(y)=\tilde{h}^3(-y)$, $n^4(y)=\tilde{h}^4(-y)$,
and $l^3(y)=\tilde{f}^3(-y)$  are all $C^1$ functions in $[0, -X_A]$.

Then,  the desired local well-posedness of the boundary value problem (\ref{gusat problem-refor-1})--(\ref{boundary condition-refor-1})
in $\tilde{A}(-X_A)$ for some $T>0$ can be obtained by the similar argument (such as transformation \eqref{fixboundary})
used in the proof  of Theorem \ref{goursatlocal} and the conclusion of Theorem \ref{thA},
which implies the desired local well-posedness in  Theorem \ref{goursatlocal-1}.

\medskip
\textbf{2}. {\it Verification of \eqref{c4-1}}.
First, from (\ref{rsnewform}), $(\ref{gusat problem})_3$, and  (\ref{boundary condition3-1}), we have
$$
r=u_0-K^{\frac{1-\gamma}{2\gamma}}_p\big(\tilde{f}^3(x)\big)^{\frac{1}{\gamma}}p^{\frac{\gamma-1}{2\gamma}}_0
\qquad\, \text{on $L_4$}.
$$
Then it follows from a direct calculation that, on $L_4$,
\begin{align}
\partial_+r=&\, r_t+cr_x=c \frac{\dd r}{\dd x}
=-\frac{c}{\gamma}\, K^{\frac{1-\gamma}{2\gamma}}_p\big(\tilde{f}^3(x)\big)^{\frac{1}{\gamma}-1}\partial_x \tilde{f}^3(x)\, p^{\frac{\gamma-1}{2\gamma}}_0\label{alongboundary-1}\\
=&-\frac{c}{\gamma}\, K^{\frac{1-\gamma}{2\gamma}}_p \hat{S}^{\frac{1}{\gamma}-1}\hat{S}_x\, p^{\frac{\gamma-1}{2\gamma}}_0.\nn
\end{align}

Furthermore, using  Lemma {\rm \ref{riclemma}}, we have
\begin{align}
\partial_+r=&\,r_t+cr_x=c\zeta+c(\zeta-\frac{1}{\gamma}\hat{S}_xa)
= 2c\zeta -\frac{c}{\gamma}\hat{S}_xa\label{alongboundary1-1} \\
=&\, 2c\zeta-\frac{1}{\gamma} K^{\frac{1-\gamma}{2\gamma}}_p c \hat{S}_x p^{\frac{\gamma-1}{2\gamma}}\hat{S}^{\frac{1}{\gamma}-1},
\nn
\end{align}
which, together with Lemma {\rm \ref{riclemma}} and  (\ref{alongboundary-1}), implies
\begin{equation}\label{relation-a}
\partial_{+}(p-p_0)=-\frac{c^2}{2\gamma}K^{\frac{1-\gamma}{2\gamma}}_p \hat{S}^{\frac{1}{\gamma}-1} \hat{S}_x
\big( p^{\frac{\gamma-1}{2\gamma}}-  p^{\frac{\gamma-1}{2\gamma}}_0).
\end{equation}

From Gronwall's inequality and $p-p_0=0$ at $F=(T,0)$, we have
$$
(p,u)=(p_0, u_0)\qquad\,\, \text{on $L_4$},
$$
which, together with Lemma {\rm \ref{riclemma}},  implies \eqref{c4-1}.

\medskip
\textbf{3}. {\it Verification of \eqref{c3-1}}.
In the rest of this proof, we denote $C\geq 1$ as a generic constant depending only on
$x_*$ and the boundary conditions on $L_3$ and $L_4$,
which may be different at each occurrence.

Similarly to the proof of \eqref{density-lower}, we can first obtain
\begin{equation}\label{density-lower-1}
C^{-1}\leq a(t,x), c (t,x)\leq C \qquad\,\,
\text{for $(t,x)\in Q(x_*)$}
\end{equation}
if $|x_*|$ is sufficiently small.

Next, using relation \eqref{tau p c}, we have
\begin{equation}\label{relation_2-1cp}
c=K^{-\frac{\gamma+1}{2\gamma}}_pK_c p^{\frac{\gamma+1}{2\gamma}}\hat{S}^{-\frac{1}{\gamma}},
\qquad a=K^{\frac{1-\gamma}{2\gamma}}_p  p^{\frac{\gamma-1}{2\gamma}}\hat{S}^{\frac{1-\gamma}{\gamma}}.
\end{equation}
Then we use  Lemmas {\rm \ref{riclemma}} and  \eqref{density-lower-1}--\eqref{relation_2-1cp} to obtain
\begin{align}
\xi=&\,-\frac{s_t}{c}=-\frac{1}{2c}(s_t-cs_x+\frac{ac}{\gamma}\hat{S}_x)
 =\frac{1}{2}\frac{\dd s}{\dd x}|_{L_3}-\frac{a}{2\gamma}\hat{S}_x\label{relation_2-1}  \\
=&\, \frac{1}{2}\Big(\partial_x h^*(x)-\frac{1}{\gamma}  K^{\frac{1-\gamma}{2\gamma}}_p
  p^{\frac{\gamma-1}{2\gamma}}\hat{S}^{\frac{1-\gamma}{\gamma}} \partial_x\tilde{f}^3(x)\Big) \qquad\,\, \text{on $L_3$}.
  \nn
\end{align}
Using the boundary assumption $\eqref{increasing entropy-1}_2$, we have
\begin{equation}\label{aupperbound}
a(F)<\frac{\gamma \min_{x\in [X_A,0]}\partial_x h^*(x)}{2\max_{x\in[X_A,0]} \partial_x\tilde{f}^3(x)}=\frac{\Lambda}{2}.
\end{equation}
Then, along $L_3$,
\begin{equation}\label{aupperbound-a}
a(y)\leq a(F) +\|a\|_{C^1(Q(x_*))}|y|\leq \frac{\Lambda}{2}+C|y|\leq \frac{3}{4}\Lambda
\end{equation}
if $y\in [x_*,0)$, and $|y|$ is small enough.

Thus,  using \eqref{density-lower-1}, \eqref{relation_2-1}, and \eqref{aupperbound-a}, we have
\begin{equation}\label{relation_2-1upperbound-1}
C\geq \xi\geq  \frac{1}{8}\min_{x\in [y,0]}\partial_x h^*(x)=C^{-1}\qquad \text{on $L_3$}.
\end{equation}

Finally, similarly to the proof of \eqref{postive}, we can also obtain
\begin{equation}\label{inside1-1}
C^{-1}\leq  \xi\leq C  \qquad \text{in $Q(x_*)$}
\end{equation}
if $|x_*|$ is sufficiently small.

\medskip
\textbf{4}. {\it Verification of  \eqref{c5-1-a}}.
Using Condition \ref{construction-1}, (\ref{k def}), and equation $(\ref{gusat problem})_3$,
we have
\begin{equation}\label{k1k2-1}
k_1>0\,\,\, \text{in $Q(x_*)$}, \,\qquad\,  \hat{S}_x,\,k_2 >0 \,\,\, \text{in $Q(x_*)\setminus (\{x=x_*\}\cup \{(T,0)\})$}.
\end{equation}
From \eqref{c4-1} and \eqref{inside1-1}, we have
$$
\zeta=0 \,\,\,\text{on $L_4$}, \,\qquad\,  \xi>C^{-1}\,\,\,\text{in $Q(x_*)$}.
$$
Then
\begin{equation}\label{partialalpha-1}
\partial_-\zeta=-k_1k_2\xi-k_1\xi^2<0 \qquad\text{on $L_4\setminus \{A,F\}$},
\end{equation}
where we have used (\ref{frem1}) and (\ref{k1k2-1}).

Furthermore, since $\xi>0$ and $\hat{S}_x\geq0$ in $Q(x_*)$,
by  (\ref{frem1}), $\zeta$ does not change from negative to positive in the  time
direction along any backward characteristic in $Q(x_*)$.
Thus, we conclude
$$
\zeta<0 \qquad \mbox{in $Q(x_*) \setminus L_4$}.
$$
\end{proof}

\begin{remark}
For verifying \eqref{c4-1}--\eqref{c5-1-a}  conveniently,
we have given the boundary conditions in Condition {\rm\ref{construction-1}}
through the  independent variable $x$, instead of $t$.
\end{remark}

\subsection{Right oriented Goursat problem in EFH}
Denote $H=(T_H,0)$, and  redefine  points $E$ and $A$ as
$$
E=(T_E, X_E)=(T_E, x_*),\quad\, A=(T_A, X_A)=(T_A, x_*).
$$
Next, we consider the right oriented Goursat problem of system \eqref{gusat problem}
in the angular domain $EFH$, where
the backward and forward characteristic
boundaries $L_3$ and $L_5$ are defined by $\eqref{free boundary-1}_1$ and
\begin{equation}\label{free boundary-2}
L_5 (EH):=\Big\{x=x^5(t),\, t\in[T_A,T]\, :\, \frac{\dd x^5(t)}{\dd t}=c(t,x^5(t)),\,x^5(T_E)=x_*\Big\}.
\end{equation}

The boundary conditions on $L_3$ and $L_5$ are given by $\eqref{boundary condition-1}_1$ and
\begin{equation}\label{boundary condition-2}
L_5 :\,  \tilde{r}(t,x^5(t))=h^5(t),\,\,\, \hat{S}(t,x^5(t))=f^3(t)\qquad\, \mbox{for $T_E\leq t \leq T_H$}
\end{equation}
with
\beq\label{c5-2}
 c(T_E,x_*)
=K_c \Big(\frac{h^3(T_E)-h^5(T_E)}{2}\Big)^{\frac{\gamma+1}{\gamma-1}}\big(f^3(T_E)\big)^{\frac{1-3\gamma}{2\gamma(\gamma-1)}}>0.
\eeq
Then, for sufficiently small $\epsilon>0$,
$x^3(t)$ (resp., $x^5(t)$) in $[T_E-\epsilon,T_E+\epsilon]$  is a strictly decreasing (resp., increasing) function with respect to $t$.
It is easy to see that the two free boundaries $x^3(t)$ and $x^5(t)$ can be rewritten as $\eqref{free boundary*-1}_1$ and
\begin{equation}\label{free boundary*-2}
L_5:\,  \frac{\dd \Omega^5(x)}{\dd x}=\frac{1}{c(\Omega^5(x),x)},\,\,\, \Omega^5(x_*)=T_E.
\end{equation}
Thus,  the  boundary conditions (\ref{boundary condition-2}) can be given by $\eqref{boundary condition2*-1}_1$ and
\begin{equation}\label{boundary condition2*-2}
L_5: \,  \tilde{r}(\Omega^5(x),x)=\tilde{h}^5(x),\,\,\, \hat{S}(\Omega^5(x),x)=\tilde{f}^3(x)\qquad\, \mbox{for $x\in[x_*,0]$},
\end{equation}
which satisfy the following condition:

\begin{condition}\label{construction-2}
$\tilde{f}^3(x)$, $\tilde{h}^3(x)$, and $\tilde{h}^5(x)$ satisfy the following properties{\rm :}
\begin{enumerate}
\item[\rm (i)] $\tilde{f}^3(x)$ satisfies the conditions in \eqref{smoothboundaryvalue-1}  for $x\in [x_*,0]$,

\smallskip
\item[\rm (ii)]
$\tilde{h}^3(x)$ and $\tilde{h}^3(x)$ satisfy
\begin{equation}\label{boundary condition3-2}
\begin{cases}
 \displaystyle
 \tilde{h}^3(x)=\big(\tilde{f}^3(x)\big)^{-\frac{1}{2\gamma}}h^{*}(x),\\[1mm]
\displaystyle
 \tilde{h}^5(x)=\big(\tilde{f}^3(x)\big)^{-\frac{1}{2\gamma}}
  \Big(u^*_0+K^{\frac{1-\gamma}{2\gamma}}_p\big(\tilde{f}^3(x)\big)^{\frac{1}{\gamma}}(p^*_0)^{\frac{\gamma-1}{2\gamma}}\Big)
\end{cases}
\end{equation}
for $x\in [x_*,0]$, where $u^*_0=u(E)$ and $p^*_0=p(E)>0$, and $h^{*}(x)$ satisfies \eqref{increasing entropy-1} with $X_A$ replaced by $x_*$.
\end{enumerate}
\end{condition}

Then we have the following local-in-time existence for the  right oriented Goursat problem:

\begin{theorem}\label{goursatlocal-2}
Let Condition {\rm\ref{construction-2}} hold.
Then there exist a constant $X_*>x_*$  and a unique $C^1$ solution $(\tilde{s}, \tilde{r},\hat{S})$
of the   right oriented Goursat problem \eqref{gusat problem} with \eqref{free boundary*-2}--\eqref{boundary condition2*-2},
or equivalently \eqref{gusat problem} with \eqref{free boundary-2}--\eqref{boundary condition-2}, in $\tilde{Q}(X_*)$ defined by
\beq\label{QQ2}
\tilde{Q}(X_*)=\big\{(t,x)\,:\, \Omega^3(x)\leq t\leq \Omega^5(x), \, x \in [x_*,X_*]\big\}.
\eeq
Furthermore,
\begin{align}
&\zeta=0\qquad \hbox{on}\ L_5, \label{c4-2}\\
&\xi>0\qquad \hbox{on}\ \tilde{Q}(X_*), \label{c3-2}\\
&\zeta>0 \qquad \text{in $\tilde{Q}(X_*)\setminus L_5$}.\label{c5-2-a}
\end{align}
Moreover, the length of the life span $|X_*-x_*|$ depends only on the $C^1$--norm of $\tilde{f}^3(x)$,
$\tilde{h}^3(x)$, and $\tilde{h}^5(x)$.
\end{theorem}

\begin{proof} We divide the proof into two steps.
\medskip

\textbf{1}. {\it Local existence of the right oriented Goursat problem in EFH}.
In order to solve the right oriented Goursat problem (\ref{gusat problem}) with \eqref{free boundary-2}--(\ref{boundary condition-2}),
we just need to reformulate it as a normal  Goursat problem under some coordinate transformations.

First, we  introduce the following new coordinates:
\begin{equation}\label{normalchange-2}
(\tau, y)=(t-T_E,x-X_E).
\end{equation}
Denote
\begin{equation}\label{rsnewform1-2}
\begin{split}
&\tilde{s}_*(\tau,y)= \tilde{s}(\tau+T_E,y+X_E),\\
&\tilde{r}_*(\tau,y)= \tilde{r}(\tau+T_E,y+X_E), \\
& \hat{S}_*(y)= \hat{S}(y+X_E).
\end{split}
\end{equation}
Then equations (\ref{gusat problem}) can be rewritten as
\begin{equation}\label{gusat problem-refor-2}
\begin{cases}
(\tilde{s}_*)_y+\frac{1}{\tilde{c}}(\tilde{s}_*)_\tau=-\frac{1}{2\gamma} \frac{(\hat{S}_*)_y}{\hat{S}_*} \tilde{r}_*,\\[1.5mm]
(\tilde{r}_*)_y-\frac{1}{\tilde{c}}(\tilde{r}_*)_\tau=-\frac{1}{2\gamma} \frac{(\hat{S}_*)_y}{\hat{S}_*} \tilde{s}_*,\\[1.5mm]
(\hat{S}_*)_\tau=0,
\end{cases}
\end{equation}
where
\begin{equation}\label{free boundary-refor-2}
\tilde{c}(\tau,y)=c(\tau+T_E,y+X_E,\tilde{s}_*, \tilde{r}_*)
=K_c \big(\frac{\tilde{s}_*-\tilde{r}_*}{2}\big)^{\frac{\gamma+1}{\gamma-1}}(\hat{S}_*)^{\frac{1-3\gamma}{2\gamma(\gamma-1)}}.
\end{equation}

The vertex for the reformulated Goursat problem is still $E$,
and the two free boundaries can be given by
\begin{equation}\label{free boundary-refor_2-2}
\begin{split}
\hat{L}_3&:\   \frac{\dd \tau^3(y)}{\dd y}=-\frac{1}{\tilde{c}(\tau^3(y),y)},\,\,\, \tau^3(0)=0,\\[1mm]
\hat{L}_5&:\   \frac{\dd \tau^5(y)}{\dd y}=\frac{1}{\tilde{c}(\tau^5(y),y)},\,\,\, \tau^5(0)=0.
\end{split}
\end{equation}
Thus, the time-space domain under consideration is
$$
\tilde{A}(-X_E)=\big\{(\tau,y)\,:\, \tau^3(y)\leq \tau \leq \tau^5(y),\,y \in [0,-X_E]\big\},
$$
and the boundary conditions are
\begin{equation}\label{boundary condition-refor-2}
\begin{split}
\tilde{L}_3&:\,\tilde{s}_*(\tau^3(y),y)=\tilde{n}^3(y),\,\, \hat{S}_*(\tau^3(y),y)=\tilde{l}^3(y) \qquad\, \text{for $y \in [0,-X_E]$},\\[1mm]
\tilde{L}_5&:\, \tilde{r}_*(\tau^5(y),y)=\tilde{n}^5(y),\,\, \hat{S}_*(\tau^5(y),y)=\tilde{l}^3(y)\qquad\, \text{for $y \in [0,-X_E]$},
\end{split}
\end{equation}
where $\tilde{n}^3(y)=\tilde{h}^3(y+X_E)$, $\tilde{n}^3(y)=\tilde{h}^3(y+X_E)$,
and $\tilde{l}^3(y)=\tilde{f}^3(y+X_E)$ are all $C^1$-functions in $[0, -X_E]$.

Then,  the desired local well-posedness of the boundary value
problem (\ref{gusat problem-refor-2})--(\ref{boundary condition-refor-2})
in $\tilde{A}(-X_E)$ for some $T>0$ can be obtained by the similar
argument (such as transformation \eqref{fixboundary})  used in the proof
of Theorem \ref{goursatlocal} and the conclusion of Theorem \ref{thA},
which  implies the desired local well-posedness  in   Theorem \ref{goursatlocal-2}.

\medskip
\textbf{2}. {\it Verification of \eqref{c4-2}--\eqref{c5-2-a}}. In the rest of the proof, we denote $C\geq 1$ as a generic constant depending only on
$T$ and the boundary conditions on $L_3$ and $L_5$,
which may be different at each occurrence.

First, the verification of \eqref{c4-2} follows from the same argument used in the proof for \eqref{c4-1}.
Next, similarly to the proof of   \eqref{density-lower},  we can  also obtain
\begin{equation}\label{densitylow3}
C^{-1}\leq a(t,x), c (t,x)\leq C\qquad\, \text{for $(t,x)\in \tilde{Q}(X_*)$}
\end{equation}
if  $|X_*-x_\epsilon |$ is  small enough, which, together with  \eqref{relation_2-1upperbound-1}
and the argument used in the proof of  \eqref{postive}, implies
$$
\, C^{-1}\leq \xi\leq C \qquad\, \text{for $(t,x)\in \tilde{Q}(X_*)$}.
$$
Finally, similarly to the proof of   \eqref{c5-1-a}, we can also show
$$
\zeta>0 \qquad\, \text{in $\tilde{Q}(X_*)\setminus L_5$}.
$$
This completes the proof.\end{proof}

It should be  pointed out that the solution in the above local existence, Theorem \ref{goursatlocal-2},
might not reach boundary $HF$, {\it i.e.}, $X_*< 0$.
If that is true, by shrinking the size of $EAF$ (which accordingly shrinks the width of $EFH$),
we could make up to boundary $HF$, since the local existence time only
depends on  the $C^1$ norm of $\tilde{f}^3(x)$, $\tilde{h}^3(x)$, and $\tilde{h}^5(x)$.
Thus, in the rest of this section, we always take $X_*=0$.

\begin{remark}
For verifying   \eqref{c4-2}--\eqref{c5-2-a} conveniently,
we have given the boundary conditions in Condition {\rm\ref{construction-2}}
through the  independent variable $x$, instead of $t$.
\end{remark}

\subsection{The de-rarefaction phenomenon}
Now we are ready to  establish the desired  continuous solution, which
can be given via  the following theorem:

\begin{theorem}\label{derarefaction}
Let $T_m>T_H>T_E>T_A>0$
$($see Fig. {\rm \ref{Fig:example4}}$)$.
Then we can construct a nonisentropic  continuous solution $(v, u, S)$ in $[0,T_m)\times \mathbb{R}$
of the Cauchy problem \eqref{lagrangian1q}--\eqref{lagrangian3q} with proper initial data,
so that $(v, u, S)(t,x)$ satisfy the following properties{\rm :}
\begin{itemize}
\item [\rm (i)]  The solution is $C^1$ for $t\in [0, T_m)$, except on the forward and backward characteristics
  passing through $E$, the forward characteristic passing through  $F$,  and lines $x=X_A$
 and $x=0$ {\rm (}i.e., only Lipschitz continuous{\rm )}{\rm ;}

\smallskip
\item [\rm (ii)]   $S_x(t,x)>0$ for $(t,x)\in [0,T_m)\times [X_E,0]$,
  and   $S_x(t,x)=0$ for $(t,x)\in [0,T_m)\times (-\infty, X_E)\cup (0,\infty)${\rm ;}

\smallskip
\item [\rm (iii)]
In the isentropic region to the left of EA for $t\in [0, T_m)$, the forward wave
is rarefaction {\rm (}$\alpha\geq 0${\rm )}, and the backward wave is a compression before the blowup {\rm (}$\beta\leq 0${\rm )}{\rm ;}
in the isentropic region to the right of HF, the forward and backward waves are both rarefaction waves {\rm (}$\alpha\geq 0$ and $\beta\geq 0${\rm )}{\rm ;}

\smallskip
\item [\rm (iv)] The solution forms a shock wave when $t=T_m$ in the isentropic region to the left of EA.
\end{itemize}
\end{theorem}

\begin{proof}
We now consider the isentropic regions to the left of EA and to the right of HF,
before some time $T^*>T_H$; see
Fig. \ref{Fig:example4}.

First, as a summary, we know that, within the life span of the continuous solution,

\smallskip
\begin{itemize}
\item In the isentropic region to the left of EA, the forward wave is rarefaction ($\alpha\geq 0$),
and the backward wave is a compression before the blowup ($\beta\leq 0$).

\smallskip
\item In the isentropic region to the right of HF, the forward wave is rarefaction ($\alpha\geq 0$),
and the backward wave is also a rarefaction ($\beta\geq 0$).
\end{itemize}

\smallskip
\noindent
Recall that, by \eqref{yq odesq}, in the isentropic region,
\[
\partial_+ \alpha=-b_2\, \alpha^2, \qquad \partial_- \beta=-b_2\, \beta^2,
\]
which gives
\[
\frac{1}{\alpha(t,x^1(t,x^1_0))}=\frac{1}{\alpha(\tau,x^1(\tau,x^1_0))}+\int_\tau^t b_2(\sigma,x^1(\sigma,x^1_0))\, \text{d}\sigma
\]
along some forward characteristic curve starting from $(0,x^1_0)$, and
\[
\frac{1}{\beta(t,x^2(t,x^2_0))}=\frac{1}{\beta(\tau,x^2(\tau,x^2_0))}+\int_\tau^t b_2(\sigma,x^2(\sigma,x^2_0))\, \text{d}\sigma
\]
along some  backward characteristic curve starting from $(0,x^2_0)$.

Thus, choosing the initial variation of $\alpha$ on $L_3$ small enough and using system \eqref{gusat problem},
it is easy to see that, for $t\in [0, T_H)$, $(\alpha, \beta)$ are small enough on EA and HF
such that $(\alpha, \beta)$ do not blow up to $\infty$ or $-\infty$
in the region to the left of EA and to the right of HF, and $s_E> r_A$.
In fact, in the extreme case that there is no wave on $L_3$,
the solution in this example becomes a stationary one,
so that $s_E>r_E=r_A$.
By the continuous dependence of the solution with respect to the initial data,
we know that  $s_E> r_A$ still holds, provided that the initial perturbation on $L_3$ is small enough.

Since $s$ and $r$ are constant along the forward and backward characteristics, respectively,
we know that, at the interaction point B of the forward characteristic BE and backward characteristic AB,
$$
s_B=s_E> r_A=r_B,
$$
which indicates that $\rho_B$ is positive.
Furthermore, it is easy to check that $\rho_B$ is the minimum value of the density to the left of EA,
which includes the isentropic backward compression and forward rarefaction, by using Lemma \ref{riclemma} in the isentropic region.
In fact, in the isentropic region, the density function is decreasing after the rarefactive wave and increasing after a compressive wave.
Then, using the standard characteristic method such as in \cite{Lidaqian}, we can prove the existence of a continuous
solution to the left of EA when $t<T^*$.
On the right-hand side of HF, the existence of a continuous solution can also be obtained, since the minimum density is $\rho_H$.
Here, in this classical existence result, the $C^1$ estimate comes from the bound of $(\alpha,\beta)$ discussed in the previous paragraph.

Finally, since the backward wave on EA is a compression which becomes a backward compression simple wave in the top-right region of BE,
it forms a shock in a finite time. Then we complete the construction.
\end{proof}

\section{Additional Remarks}
In this section, we make two important  remarks on a better density lower bound estimate and the shock formation.

\subsection{A better estimate on the lower bound of the density for the $p$--system}

We now discuss another possible condition on the initial data with a far-field vacuum,
which also leads to the proof of the formation of singularities.
The new condition on the initial data is \eqref{apenc}, which takes the role of \eqref{yq1}.
We also keep all the other initial assumptions in Condition \ref{initialdata2}.

For simplicity, we consider only the isentropic case; a similar new result
for the nonisentropic case can be obtained.
The initial condition \eqref{apenc} is more restrictive than \eqref{yq1}
in Theorem \ref{p_sing_thm} on the far-field density;
however, a more accurate estimate of the life span of $C^1$ solutions
can be provided when such a new initial condition is satisfied.

The result we prove here is the same as Theorem \ref{p_sing_thm},
with only change of replacing \eqref{yq1} by \eqref{apenc}.
On the other hand, if \eqref{apenc} is satisfied, we can
obtain a better lower bound estimate of the density,
in the order of $O((1+t)^{-1})$ for any $\gamma>1$ in the interior (not near the far-field),
than the bound in Lemma \ref{density_low_bound_1-3}.
This estimate leads to a better estimate of the life span of $C^1$ solutions.

Since the proof of the new result is almost the same as that of Theorem \ref{p_sing_thm},
except the interior lower bound estimate of the density, we show only this part here.

\begin{lemma} \label{property1}
Let \eqref{apenc} hold.
Then
\beq\label{al-be-tq1}
\max\big\{\sup_{(t,x)\in [0, T)\times \mathbb{R}}s_x(t,x), \sup_{(t,x)\in [0, T)\times \mathbb{R}}r_x(t,x)\big\}<M.
\eeq
\end{lemma}

Once we prove this lemma, then, by \eqref{p1} and \eqref{p_sr},
\[
v_t=u_x=\frac{1}{2}(s_x+r_x)
\]
so that
\[
v(t,x)< v(0,x)+Mt,
\]
which provides an interior lower bound of the density in the order of $O((1+t)^{-1})$.
Then, similarly to the proof of Theorem \ref{p_sing_thm},
we can obtain the necessary and sufficient result on the formation of singularities.

\begin{remark}
In particular, if
$$
\max\big\{\sup_{x\in\mathbb R}s_x(0,x), \sup_{x\in\mathbb R}r_x(0,x)\big\}\leq 0,
$$
then it follows from the definition of $(\alpha,\beta)$ and the Riccati equations \eqref{p_y_eq-1}
that
$$
\max\big\{\sup_{(t,x)\in [0, T)\times \mathbb{R}}s_x(t,x), \sup_{(t,x)\in [0, T)\times \mathbb{R}}r_x(t,x)\big\}\leq 0,
$$
which implies that
\[
\rho(t,x)>\rho(0,x) \qquad\,\, \text{for $(t,x)\in (0,T]\times \mathbb{R}$}.
\]
\end{remark}

We now give the proof of Lemma \ref{property1} here for self-containedness.

\medskip
\noindent
{\it Proof of Lemma {\rm \ref{property1}}}.
We  prove  (\ref{al-be-tq1}) by contradiction. Without loss of generality, assume that
$s_x(t_0,x_0)=M$ at some point $(t_0,x_0)$; see Fig. \ref{fig0q}.

\begin{figure}[htp] \centering
		\includegraphics[width=.4\textwidth]{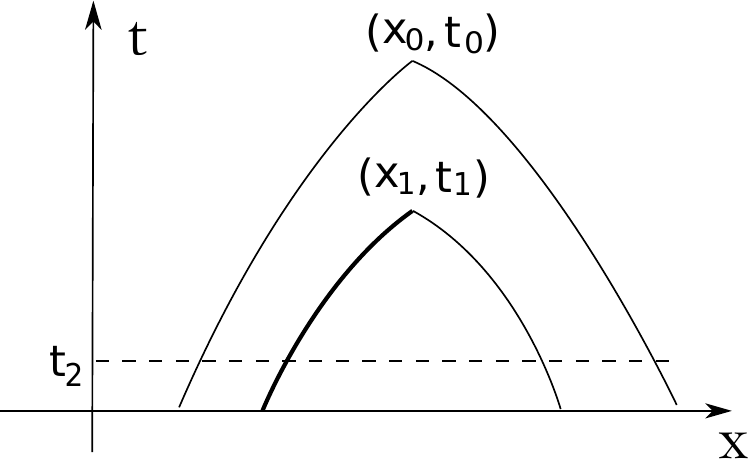}
		\caption{Proof of Lemma {\rm \ref{property1}}.\label{fig0q}}
\end{figure}

Since the wave speed $c$ is bounded above,
we can find the characteristic triangle with vertex $(t_0,x_0)$ and the lower boundary
on the initial line $t=0$, denoted by $\Omega$.
Then we can find the first time $t_1$ such that $s_x=M$ or $r_x=M$ in $\Omega$.
More precisely,
\[
\max\big\{\sup_{(t,x)\in \Omega, t<t_1}s_x(t,x), \sup_{(t,x)\in \Omega, t<t_1}r_x(t,x)\big\}<M,
\]
and $s_x(t_1,x_1)=M$ and/or $r_x(t_1,x_1)=M$ for some $(t_1,x_1)\in\Omega$.
Without loss of generality, we still assume that
$s_x(t_1,x_1)=M$.
The proof for the other case is the same.

Denote the characteristic triangle with vertex $(t_1,x_1)$ as $\Omega_1\subset\Omega$. Then
\beq\label{p_main_1}
\max\big\{\sup_{(t,x)\in \Omega_1, t<t_1}s_x(t,x), \sup_{(t,x)\in \Omega_1, t<t_1}r_x(t,x)\big\}<M,
\eeq
and $s_x(t_1,x_1)=M$.
By the continuity of $s_x$,  there exists $t_2\in[0,t_1)$ such that
\beq\label{p_main_2}
s_x(t,x)>0 \qquad\, \hbox{for any $(t,x)\in \Omega_1$ and $t\geq t_2$}.
\eeq

Set
\beq\label{kK1q}
K_3:=\max_{(t,x)\in  [0,T)\times \mathbb{R}}k_1(t,x),
\eeq
where $k_1(t,x)$ is defined in \eqref{k def},
and $K_3$ is a constant depending only on $\gamma$ and the initial condition.

\smallskip
Next we derive a contradiction.
By \eqref{frem1} and \eqref{p_main_1}--\eqref{kK1q},
along the forward characteristic segment through $(t_1,x_1)$ for $t_2\leq t<t_1$,
\begin{eqnarray*}
\partial_+s_x=k_1\big(s_x r_x-s_x^2\big)\leq  k_1 s_x(M-s_x) \leq K_3 s_x(M-s_x),
\end{eqnarray*}
which gives, through the integration along the characteristic,
\begin{eqnarray*}
\frac{1}{M}\ln (\frac{s_x(t)}{M-s_x(t)})\leq \frac{1}{M}\ln(\frac{s_x(t_2)}{M-s_x(t_2)})+K_3 (t-t_2).
\end{eqnarray*}
As $t\rightarrow t_1-$, the left-hand-side term tends to infinity,
while the right-hand-side terms tend to a finite number,
which gives a contradiction.
This implies that (\ref{al-be-tq1}) holds; that is, $s_x$ and $r_x$
are uniformly bounded above.
This completes the proof.

\subsection{Shock formation of the $p$--system}
Based on the conclusions obtained in Theorems \ref{p_sing_thm}--\ref{Thm singularity2},
a natural question is that, for the compressible perfect fluid,
what types of singularities the solution may form  due to the initial compression,
that is, how it grows out of $C^1$ solutions.
A widely accepted answer for this question for the one-dimensional case is that
the formed singularity is a shock wave.

There are many results related to the shock formation.
We now review some existing results on addressing this question.
In fact, based on the conclusions obtained in \cite{ Chenshuxing,  Chenshuxing2, Kongdexing,lebaud},
a fairly clear description on the shock formation mechanism of the $p$--system has been provided.
We also mention some related results on the nonisentropic flow.
In \cite{Chris}, this problem has also been considered by solving a free boundary
problem in a neighborhood of the blowup point.

For $\gamma>1$, in \cite{ Chenshuxing, lebaud},
it has been shown that the singularity caused by the initial compression is actually a shock,
under some additional information of the solution at the first blowup point.

First, under the hypothesis that one of the Riemann invariants is a constant,
the shock formation has been studied by Lebaud \cite{lebaud}:

\begin{theorem}[\cite{lebaud}]\label{leb}
Let $\gamma>1$, $r_0\in C^4$,  $s_0(x)=\overline{s}$  for $x\in \mathbb{R}$ with some constant $\overline{s}$,
$g(x)=-c(\overline{s},r_0(x))$, and $g'(x)$ takes its global minimum at $x_0$ with
\begin{equation}\label{lebaundcondition}
g'(x_0)<0,\quad g{''}(x_0)=0, \quad g{'''}(x_0)>0.
\end{equation}
Then  the  Cauchy problem \eqref{p1}--\eqref{p2} with \eqref{initial1.2}
admits a  weak entropy solution, which is smooth in $[0,t_0)$ and continuous in  $[0, t_0]$ and
can be extended as a weak entropy solution to $t>t_0$,
and the solution has a shock $x= \varphi(t)$ starting from  $(t_0,x_0)$.
\end{theorem}

\begin{remark}
From the definition of $g(x)$,
$g'(x_0)<0$ is equivalent to $r'_0(x_0)<0$.
Furthermore, for the Cauchy problem \eqref{p1}--\eqref{p2} with \eqref{initial1.2},
one can choose the characteristic coordinates $(a, b)$ in the following way {\rm (see \S 4 of Chapter 2 in \cite{Lidaqian}):}
On any fixed forward $($resp., backward$)$ characteristic, a $($resp., b$)$ is taken as the $x$-coordinate of
the intersection point of this characteristic with the $x$-axis.
Let  $x(t;b)$ be the  backward characteristic passing through point $(0,b)$. Then,
under the assumption that $s_0(x)=\overline{s}$, condition \eqref{lebaundcondition} is equivalent to
\begin{align}
&x_b(t_0;b_0)=0,\label{condition-1}\\
& x_{bb}(t_0;b_0)=0, \quad x_{bbb}(t_0;b_0)>0,\label{condition-2}
\end{align}
at the first blowup point $(t_0,x_0)$ with   $ x_0=x(t_0;b_0)$ $($see Lemma 3.2 of Kong \cite{Kongdexing}$)$.

For general smooth data $(s_0(x), r_0(x))\in C^4$,
if only one Riemann invariant blows up, and the blowup point $(t_0,x_0)$ is formed by the normal
squeeze ({\it i.e.}, satisfying conditions \eqref{condition-1}--\eqref{condition-2}) of only one family
of characteristics, while another family of characteristics does not squeeze at the same point,
the problem of formation and construction of a shock wave was also established
in Chen-Dong \cite{Chenshuxing}.
\end{remark}

\begin{remark}
By employing the method of characteristic coordinates \cite{Lidaqian} and the singularity theory of smooth mappings \cite{Whitney},
a similar result has been established for general $2\times 2$ quasilinear hyperbolic systems in   \cite{Kongdexing}.
\end{remark}

However, in general, condition \eqref{condition-2} near the first blowup point
is difficult to verify.
Even the time and position of the first blowup point are also difficult to obtain
for general systems of hyperbolic conservation laws.

\medskip
Finally, we present a related result on the nonisentropic flow.
Assume that  the initial data $(\rho_0, u_0,S_0)$ satisfy
$$
\rho_0=\overline{\rho}+\epsilon f_1(x), \quad  u_0=\epsilon f_2(x), \quad S_0=\overline{S}+\epsilon f_3(x),
$$
where  $\epsilon>0$, $\overline{\rho}>0$, and $\overline{S}$ are all constants, and
$f_i(x)$, $i=1,2,3$, are smooth functions with compact support.
Denote
\begin{equation*}
\begin{split}
&P(\rho,S)=Ke^{\frac{S}{c_v}}\rho^\gamma,\quad\, c^2_E(\rho,S)=\partial_\rho P(\rho,S),\quad\, \overline{c}_E=c_E(\overline{\rho},\overline{S}),\\
&h_1(x)=\frac{\partial_\rho c_E(\overline{\rho},\overline{S})+\overline{c}_E}{2\overline{c}_E}\big(f'_2(x)-\frac{\overline{c}_E}{\overline{\rho}}f'_1(x)\big),\\
&h_2(x)=\frac{\partial_\rho c_E(\overline{\rho},\overline{S})+\overline{c}_E}{2\overline{c}_E}\big(f'_2(x)+\frac{\overline{c}_E}{\overline{\rho}}f'_1(x)\big).
\end{split}
\end{equation*}

Assume that
$$
M_1=\min h_1(x)<M_2=\min h_2(x),
$$
and $h_1(x)$ has a unique strictly negative quadratic minimum.
Then, as shown in Chen-Xin-Yin \cite{Chenshuxing2},
for small $\epsilon>0$,
the Cauchy problem \eqref{lagrangian1q}--\eqref{lagrangian3q} and \eqref{initial1} admits a weak entropy solution,
which is smooth in $[0,T_\epsilon)$ and continuous in  $[0,T_\epsilon]$,
and has a unique shock $x=\varphi(x)$ starting from the unique blowup point $(T_\epsilon,x_\epsilon)$ in $(T_\epsilon,T_\epsilon+1)$.

\medskip
\appendix
\section{Local-in-Time Well-Posedness of a Boundary Value Problem for Nonlinear Hyperbolic Systems}

In this appendix, we consider the local-in-time well-posedness of the  following boundary value problem in some angular domain:
\begin{equation}\label{typicalboundaryproblem}
\begin{cases}
(U_i)_t+\kappa_i(t,x,U)(U_i)_x=\phi_i(t,x,U) \qquad \text{for $(t,x)\in R(\delta_0)$ and  $i=1,2$}, \\[1mm]
U_1=G_1(t)  \qquad \text{on $x=t$},\\[1mm]
U_2=G_2(t)  \qquad \text{on $x=0$},
\end{cases}
\end{equation}
where $U=(U_1,U_2)^\top$,
$$
R(\delta_0):=\big\{(t,x)\; : \; 0\leq t \leq \delta_0\leq 1, \,\, 0\leq x\leq t\big\},
$$
$\kappa_1$, $\kappa_2$, $\phi_1$, and $\phi_2$ are assumed to be the functions  of $(t,x,U)$,
and  $G_1$ and $G_2$ are assumed to be the functions of  $t$.

\subsection{Definitions of some functionals}
To make the corresponding statement and proof precise, we first define some required functionals.

\smallskip
Let $f(t,x)$ be a function on a bounded domain $\Omega$.
The modulus of continuity of $f(t,x)$ is defined by the following nonnegative function:
\begin{equation*}
w(\epsilon | f)
:=\sup_{\begin{subarray}{c} |t_1-t_2|\leq  \epsilon \\ |x_1-x_2|\leq  \epsilon \\ (t_1,x_1), (t_2,x_2) \in \Omega \end{subarray}}
  |f(t_1,x_1)-f(t_2,x_2)|\qquad\,\,\, \mbox{for $0\leq \epsilon <\infty$}.
\end{equation*}
The  modulus of continuity of a vector function $\Xi=(\Xi_1,...,\Xi_k)^\top$ or a set of functions $\Gamma=\{\varrho\}$
on the bounded domain $\Omega$ can be  defined by
$$
w(\epsilon | \Xi):=\max_{i=1,...,k} w(\epsilon | \Xi_i),\qquad w(\epsilon | \Gamma)=\max_{\varrho\in \Gamma } w(\epsilon | \varrho).
$$
Moreover, we  define the  modified modulus of continuity of $f(t,x)$ by
\begin{equation*}
\Lambda(\epsilon | f):=\sup_{\begin{subarray}{c}  |t_1-t_2|\leq  \epsilon \\  0\leq \frac{x_1-x_2}{t_1-t_2}\leq 1 \\ (t_1,x_1), \ (t_2,x_2) \in \Omega \end{subarray} }
|f(t_1,x_1)-f(t_2,x_2)|\qquad\,\,\, \mbox{for $0\leq \epsilon <\infty$}.
\end{equation*}
Similarly,
$$
\Lambda(\epsilon | \Xi):=\max_{i=1,...,k} \Lambda(\epsilon | \Xi_i).
$$
Then, by direct calculations, we see that the modified modulus of continuity satisfies the following properties:
\begin{lemma}\label{fujiaxingzhi}
Let $\Omega=R(\delta)$ for some $\delta >0$. Then
\begin{equation*}
\begin{split}
&\Lambda((\epsilon_1+\epsilon_2) | f)\leq  \Lambda(\epsilon_1 | f)+\Lambda(\epsilon_2 | f),\\[0.5mm]
&\Lambda(\epsilon| fg)\leq  \|f\|_{C^0(R(\delta))} \Lambda(\epsilon | g)+ \|g\|_{C^0(R(\delta))} \Lambda(\epsilon | f),\\[0.5mm]
\displaystyle
& \Lambda(\epsilon| f(g(t,x)))\leq  w(\Lambda(\eta |g) | f),\\[0.5mm]
&\Lambda\big(\epsilon\big| f^{-1}\big)\leq  C^{-2} \Lambda(\epsilon | f) \qquad \mbox{if $|f(t,x)|\geq C>0$},\\[0.5mm]
&C_1w(\epsilon | f)\leq  \Lambda(\epsilon| f)  \leq  C_2w(\epsilon | f),\\[0.5mm]
&\Lambda(\epsilon| \int_{f_1(t,x)}^{f_2(t,x)} f(\tau; t,x)\dd\tau)
 \leq  C_3\big(\Lambda(\epsilon | f_1)+\Lambda(\epsilon | f_2)\big)
   +\int_0^\delta \Lambda(\eta| f(\tau; t,x))\,\dd\tau,
\end{split}
\end{equation*}
where $\epsilon, \epsilon_1,  \epsilon_2$, and $C$ are all positive constants,
$f, g, f_1$, and $f_2$ are functions on  $\Omega$,   $C_1$ and $C_2$ are positive constants independent of $f$,
and $C_3>0$ is a constant satisfying $|f(\tau; t,x)|\leq C_3 $
in $\{(t,x)\in R(\delta)\,:\, 0\leq f_1(t,x)\leq \tau \leq f_2(t,x)\leq \delta\}$.
\end{lemma}

\medskip
Define
\begin{equation*}\begin{split}
&\Sigma(\delta):= \big\{\psi(t,x)=(\psi_1,\psi_2)^\top\,:\,\psi\in C^1(R(\delta)), \, \psi(0,0)=0\big\},\\[2pt]
&\displaystyle
\Sigma(\delta | M_1):=\big\{\psi(t,x)\,:\, \psi\in \Sigma(\delta), \, \|\chi\|_{C^0(R(\delta))}\leq M_1\big\},
\end{split}
\end{equation*}
where
$\chi=(\chi_1, \chi_2,\chi_3,\chi_4)^\top$
with
$$
\chi_i=(\psi_i)_t, \quad \chi_{i+2}=(\psi_i)_t+( \psi_i)_x \qquad\,\, \mbox{for $i=1,2$}.
$$

\smallskip
Finally, for $\psi\in C^1(R(\delta))$, define
\begin{equation*}
\begin{cases}
\tilde{\kappa}_i=\kappa_i(t,x,\psi(t,x)),\quad \tilde{\phi}_i=\phi_i(t,x,\psi(t,x)) \qquad\,\, \mbox{for $i=1,2$},\\[4pt]
\displaystyle
\Gamma(\psi)=\big\{\tilde{\kappa}_1, \tilde{\kappa}_2, (\tilde{\kappa}_1)_x,  (\tilde{\kappa}_2)_x,    \tilde{\phi}_1,
              \tilde{\phi}_2, (\tilde{\phi}_1)_x,  (\tilde{\phi}_2)_x,  (1-\tilde{\kappa}_1(t,t))^{-1}, \tilde{\kappa}_2(t,0)^{-1} \big\}.
\end{cases}
\end{equation*}
For simplicity, in the rest of this appendix, we denote
\begin{equation*}
\begin{split}
&\|f(t,x)\|_\delta:=\|f\|_{C^0(R(\delta))},\quad  \|f(t)\|_\delta:=\|f\|_{C^0([0,\delta])}, \\[1mm]
& \|\Gamma(\psi)\|_\delta:=\max_{f\in \Gamma(\psi)}\|f\|_{\delta}, \quad G=(G_1,G_2)^\top,\quad \phi=(\phi_1, \phi_2)^\top.
\end{split}
\end{equation*}

\subsection{Conditions on  $G_i$, $\kappa_i$, and $\phi_i$ for $i=1,2$}
We assume that  $G_i$, $\kappa_i$, and $\phi_i$
satisfy the following conditions:

\smallskip
\begin{enumerate}
\item[\rm (i)] $G_1(t)$ and $G_2(t)$  are  $C^1$ functions of  $t$ in $ [0, \delta_0]$;

\item[\rm (ii)] $\kappa_i$, $(\kappa_i)_x$, $(\kappa_i)_{U_j}$, $\phi_i$, $(\phi_i)_x$,
and $(\phi_i)_{U_j}$, $i,j=1,2$,
are continuous functions of   $(t,x,U)^\top$ in $ R(\delta_0)\times \mathbb{R}^2$;

\item[\rm (iii)]  On $R(\delta_0)$, for any $\psi\in  \Sigma(\delta_0 | M_1)$,
\begin{equation}\label{coeff8}
\kappa_1(t,x,\psi)|_{x=0}= 0, \qquad  \kappa_2(t,x,\psi)|_{x=t}= 1.
\end{equation}
\end{enumerate}

\subsection{Local-in-time well-posedness theorem and its proof}

Now we state the main theorem of this appendix, which is a simplified version
of Theorem 6.1 in \cite[Chapter 2]{liyu}.
For self-containedness and convenience of the reader, we also give a proof here.

\begin{theorem}\label{thA}
Let conditions {\rm (i)}--{\rm (iii)} in \S {\rm A.2} hold.
Then the boundary value  problem \eqref{typicalboundaryproblem}
admits a unique $C^1$ solution $U=U(t,x)$ on $R(\delta)$ for sufficiently small $\delta>0$.
\end{theorem}

\begin{proof} We divide the proof into four steps.

\medskip
\noindent
\textbf{1}. {\it Linearization}.  Without loss of generality, we can assume that $U(0,0)=0$;
otherwise, it can be achieved by a translation transformation.  Moreover, we  choose  $M_1$  sufficiently large such that
\begin{equation}\label{jjjjj}
\|G_t(t)\|_{\delta_0}+\sup_{ \substack{(t,x)\in R(\delta_0), \\ |\psi|\leq 1}}|\phi(t,x,\psi)|\leq \frac{1}{2} M_1.
\end{equation}
By definition, for any  $\psi\in \Sigma(\delta | M_1)$,  $\|\psi\|_\delta\leq M_1 \delta $.

Now, for any $\psi\in \Sigma(\delta_0 | M_1)$, we consider the following linear problem:
\begin{equation}\label{linearproblem}
\begin{cases}
(U_i)_t+\tilde{\kappa}_i(t,x)(U_i)_x=\tilde{\phi}_i(t,x) \qquad \text{for $(t,x)\in R(\delta_0)$ and  $i=1,2$}, \\[1mm]
U_1=G_1(t)  \qquad \text{on $x=t$},\\[1mm]
U_2=G_2(t)  \qquad \text{on $x=0$}.
\end{cases}
\end{equation}

Let $X_i(\tau;t,x)$ for $\tau\leq t$ be the $i$-th characteristic curve of the system in \eqref{linearproblem}
which passes through point $(t,x)\in R(\delta_0)$ for $i=1,2$. Thus, for $i=1,2$,
\begin{equation}\label{curve--1}
\begin{cases}
\frac{\dd X_i(\tau;t,x)}{\dd\tau}=\tilde{\kappa}_i(\tau,X_i(\tau;t,x)),\\
X_i(t;t,x)=x.
\end{cases}
\end{equation}
Then it follows from \eqref{coeff8} that there exists a positive constant  $\delta_1\leq \delta_0$
such that  $X_1(\tau;t,x)$ and $X_1(\tau;t,x)$ are $C^1$ functions of $(\tau, t,x)$ and

\smallskip
\begin{enumerate}
\item[\rm (H1)] $\tilde{\kappa}_1(t,x)$ and $\tilde{\kappa}_2(t,x)$ satisfy
\begin{equation}\label{tezhengqulv}
\begin{split}
&\tilde{\kappa}_1(t,0)=0,\quad \tilde{\kappa}_2(t,t)=1 \qquad\,\, \text{for $t\in [0,\delta_1]$},\\
&\tilde{\kappa}_1(t,x)<1, \quad \tilde{\kappa}_2(t,x)>0 \qquad \text{for $(t,x)\in R(\delta_1)$};
\end{split}
\end{equation}

\item[\rm (H2)] The first  (resp., second) characteristic curve $X_1(\tau;t,x)$ (resp., $X_2(\tau;t,x)$) intersects
the boundary line $x=t$ (resp., $x=0$) at one and only one point $P_1=(\tau_1(t,x),y_1(t,x))$ (resp., $P_2=(\tau_2(t,x),y_2(t,x))$),
and $\tau_i(t,x)$ and $y_i(t,x)$ ($i=1,2$) are also $C^1$ functions on $R(\delta_1)$ and satisfy
\begin{equation}\label{tezhengqulv2}
y_1(t,x)=\tau_1(t,x)=X_1(\tau_1(t,x);t,x),\,\,\,
y_2(t,x)=0=X_2(\tau_2(t,x);t,x);
\end{equation}

\smallskip
\item[\rm (H3)] Each characteristic arc
 \begin{equation}\label{tezhengqulv3}
Q_i=\big\{y=X_i(\tau;t,x)\,:\, \tau_i(t,x)\leq \tau \leq t, (t,x)\in R(\delta_1)\big\}
\end{equation}
lies wholly on $R(\delta_1)$ for each $i=1,2$.
\end{enumerate}
These imply that
there exists a positive constant  $\delta_2\leq \delta_1$ such that
\begin{equation}\label{coeff11}
\|\psi\|_{\delta_2}\leq M_1 \delta_2\leq 1 \qquad\,\, \text{for any $\psi\in \Sigma(\delta_2|M_1)$},
\end{equation}
and  problem \eqref{linearproblem} has a unique $C^1$ solution $U$ in $R(\delta_2)$ satisfying
\begin{equation}\label{solutionform}
U_i(t,x)=G_i(\tau_i(t,x))+\int_{\tau_i(t,x)}^t \tilde{\phi}_i(\tau,X_i(\tau;t,x) )\,\text{d}\tau
\qquad\mbox{for $i=1,2$}.
\end{equation}
In this way, we determine an operator
\begin{equation}\label{solutionoperator}
U=\Phi \psi,
\end{equation}
which maps $\Sigma(\delta_2 | M_1)$ into $\Sigma(\delta_2)$.
We now show that operator $\Phi$ actually maps  $\Sigma(\delta_2 | M_1)$ into itself
and that, if $\delta_2$ is sufficiently small, $\Phi$ possesses a unique fixed point
that is the solution of the original problem \eqref{typicalboundaryproblem}.

\medskip
\noindent
\textbf{2}. {\it The apriori estimates of the linear problem}.
In order to establish the desired estimates, we need to estimate the first order derivatives of $U$.
First, formally differentiating each side of \eqref{linearproblem}  with respect to $x$,
we have
\begin{equation}\label{solutionform-1}
(-\varphi_i+\varphi_{2+i})_t+\tilde{\kappa}_i (-\varphi_i+\varphi_{2+i})_x=(\tilde{\phi}_i)_x-(\tilde{\kappa}_i)_x(-\varphi_i+\varphi_{2+i}),
\end{equation}
where $\varphi=(\varphi_1, \varphi_2,\varphi_3,\varphi_4)^\top$ with $\varphi_i=(U_i)_t$
and $\varphi_{i+2}=(U_i)_t+( U_i)_x$ for $i=1,2$.
Integrating each side of equation \eqref{solutionform-1} along $Q_i$, we obtain\footnote{\eqref{solutionform-2} can be
strictly derived by a standard approximation on $u$ (with respect to the $C^1$--norm) by a series of $C^2$ functions on $R(\delta_2)$.}
\begin{equation}\label{solutionform-2}
\begin{split}
(-\varphi_i+\varphi_{2+i})(t,x)
=&\, (U_i)_x(\tau_i(t,x),y_i(t,x))\\
&+\int_{\tau_i(t,x)}^t\big((\tilde{\phi}_i)_x-(\tilde{\kappa}_i)_x(-\varphi_i+\varphi_{2+i})\big)(\tau,X_i(\tau;t,x))\,\text{d}\tau.
\end{split}
\end{equation}

Then differentiating the boundary conditions in \eqref{linearproblem} on $x=t$ yields
\begin{equation}\label{solutionform-3}
(U_1)_t+(U_1)_x=(G_1)_t,
\end{equation}
which, along with the equations in \eqref{linearproblem}, leads to
\begin{equation}\label{solutionform-4}
(U_1)_x=\frac{(G_1)_t-\tilde{\phi}_1}{1-\tilde{\kappa}_1} \qquad \text{on $x=t$}.
\end{equation}
It follows from \eqref{solutionform-2} and \eqref{solutionform-4} that
\begin{equation}\label{solutionform-5}
(-\varphi_1+\varphi_{3})(t,x)=J_1(t,x),
\end{equation}
where
\begin{equation}\label{solutionform-55}
\begin{split}
J_1(t,x)=&\,\frac{1}{1-\tilde{\kappa}_1(\tau_1(t,x),y_1(t,x))}\big((G_1)_t(\tau_1(t,x))-\tilde{\phi}_1(\tau_1(t,x),y_1(t,x))\big)\\
&\, +\int_{\tau_1(t,x)}^t\big((\tilde{\phi}_1)_x-(\tilde{\kappa}_1)_x(-\varphi_1+\varphi_{3})\big)(\tau,X_1(\tau;t,x) )\,\text{d}\tau.
\end{split}
\end{equation}
Similarly, we have
\begin{equation}\label{solutionform-6}
(-\varphi_2+\varphi_{4})(t,x)=J_2(t,x),
\end{equation}
where
\begin{equation}\label{solutionform-66}
\begin{split}
J_2(t,x)=&\,\frac{1}{\tilde{\kappa}_2(\tau_2(t,x),y_2(t,x))}\big(\tilde{\phi}_2(\tau_2(t,x),y_2(t,x))-(G_2)_t(t)\big)\\
&\, +\int_{\tau_2(t,x)}^t\big((\tilde{\phi}_2)_x-(\tilde{\kappa}_2)_x(-\varphi_2+\varphi_{4})\big)(\tau,X_2(\tau;t,x) )\,\text{d}\tau.
\end{split}
\end{equation}

On the other hand, it follows directly from the equations in \eqref{linearproblem} that
\begin{equation}\label{solutionform-7}
(1-\tilde{\kappa}_i)\varphi_i+\tilde{\kappa}_i\varphi_{2+i}=\tilde{\phi}_i \qquad\,\,\mbox{for $i=1,2$},
\end{equation}
which, along with \eqref{solutionform-5} and \eqref{solutionform-6}, implies that
\begin{equation}\label{solutionform-8}
\varphi_i(t,x)=\tilde{\phi}_i(t,x)-(\tilde{\kappa}_iJ_i)(t,x),\quad \varphi_{2+i}(t,x)=\tilde{\phi}_i(t,x)+\big((1-\tilde{\kappa}_i)J_i\big)(t,x).
\end{equation}

Based on the formulas in  \eqref{solutionform-8}, we have the following lemma.
\begin{lemma}\label{estimate-1}
There exists a positive constant  $\delta_3\leq \delta_2$ such that, for any $0<\delta\leq \delta_3$,
operator $\Phi$ actually maps  $\Sigma(\delta | M_1)$ into itself.
\end{lemma}

\begin{proof}
Hereinafter, $E_1=E_1(M_1)\geq 1$  denotes a generic
constant that depends only on $M_1$, independent of the choice of $\psi$.

First, for any $\psi\in  \Sigma(\delta_2| M_1)$,  one has
\begin{equation}\label{solutionform-9}
\|\Gamma(\psi)\|_{\delta_2}\leq E_1.
\end{equation}
It follows from \eqref{solutionform} that
\begin{equation}\label{solutionform-10}
\|U\|_{\delta_2}\leq  \|G\|_{\delta_2}+\delta_2\|\tilde{\phi}\|_{\delta_2}
\leq  E_1 (\|G_t\|_{\delta_2}+1)\iota(\delta_2),
\end{equation}
where
$\iota(\eta)=\max_{i=1,2}\big\{\omega(\eta|\tilde{\kappa}_i),\omega(\eta|\tilde{\phi}_i),\eta\big\}$,
and the  domain  of $(t,x)$ is  $R(\delta_2)$.

Denote
\begin{equation}\label{solutionform-1aa}
\sigma(\eta):=\max_{i=1,2}\big\{\omega(\eta|\kappa_i(t,x,\psi)),\omega(\eta|\phi_i(t,x,\psi)),\eta\big\},
\end{equation}
where the domain of $(t,x,\psi)$ is $\{(t,x)\in R(\delta_2)\,:\, |\psi|\leq 1\}$. It is direct to see that
\begin{equation}\label{solutionform-11bb}
\iota(\eta)\leq E_1 \sigma(\eta) \qquad \text{for any $\psi\in \Sigma(\delta_2|M_1)$}.
\end{equation}
Moreover, on $R(\delta_2)$,
\begin{equation}\label{solutionform-12}
|t-\tau_i(t,x)|\leq t \leq \delta_2,\quad |x-y_i(t,x)|\leq \|\tilde{\kappa}\|_{\delta_2} |t-\tau_i(t,x)|\leq \|\tilde{\kappa}\|_{\delta_2}\delta_2,
\end{equation}
which, along with  \eqref{tezhengqulv}, yields that
\begin{align}
&\Big\|\frac{1-\tilde{\kappa}_1(t,x)}{1-\tilde{\kappa}_1(\tau_1,y_1)}\Big\|_{\delta_2} +\Big\|\frac{1}{1-\tilde{\kappa}_1(\tau_1,y_1)}\Big\|_{\delta_2} +\Big\|\frac{\tilde{\kappa}_2(t,x)}{\tilde{\kappa}_2(\tau_2,y_2)}\Big\|_{\delta_2}  \leq  1+E_1\iota(\delta_2),
  \label{solutionform-13a}\\[1mm]
&\Big\|\frac{\tilde{\kappa}_1(t,x)-\tilde{\kappa}_1(\tau_1,y_1)}{1-\tilde{\kappa}_1(\tau_1,y_1)}\Big\|_{\delta_2} +\Big\|\frac{\tilde{\kappa}_2(t,x)-\tilde{\kappa}_2(\tau_2,y_2)}{\tilde{\kappa}_2(\tau_2,y_2)}\Big\|_{\delta_2} \leq  E_1\iota(\delta_2),
\label{solutionform-13b}\\[1mm]
& \Big\|\frac{-\tilde{\kappa}_1(t,x)}{1-\tilde{\kappa}_1(\tau_1,y_1)}\Big\|_{\delta_2}
+\Big\|\frac{1-\tilde{\kappa}_2(t,x)}{\tilde{\kappa}_2(\tau_2,y_2)}\Big\|_{\delta_2} \leq E_1\iota(\delta_2).\label{solutionform-13c}
\end{align}

Then it follows from \eqref{solutionform-55}, \eqref{solutionform-66}, \eqref{solutionform-8},
\eqref{solutionform-11bb}--\eqref{solutionform-13c}, and the Gronwall inequality  that there exists a positive
constant $\delta_3\leq \delta_2$ such that, if $\delta\leq \delta_3$, for any $\psi\in \Sigma(\delta|M_1)$, $U=\Phi \psi$ satisfies the estimate:
\begin{equation}\label{solutionform-14}
\|\varphi\|_{\delta}\leq  \big(1+E_1 \sigma(\delta)\big) \big(\|G_t\|_{\delta}+\|\tilde{\phi}\|_{\delta}\big)
          +E_1\sigma(\delta),
\end{equation}
which, along with \eqref{jjjjj}, gives the desired conclusion.
\end{proof}

\noindent
\textbf{3}. {\it The solution operator's contraction of the linear problem}.
We first show the contraction of $\Phi$ with respect to the  $C^0$--norm.

\begin{lemma}\label{contraction}
There exists a positive constant  $\delta_4\leq \delta_3$ such that, for any $0<\delta\leq \delta_4$,
$\Phi$ is a contraction operator with respect to the $C^0$--norm from $\Sigma(\delta | M_1)$ into itself.
That is, if $\psi',\psi''\in \Sigma(\delta | M_1)$, then $U'=\Phi \psi'$, $U''=\Phi  \psi''$, and
there exists some constant $\mu\in [0,1)$ such that
\begin{equation}\label{solutionform-16}
\|U'-U''\|_\delta\leq \mu \|\psi'-\psi''\|_\delta \qquad \mbox{on $R(\delta)$}.
\end{equation}
\end{lemma}

\begin{proof} Denote $\overline{U}=U'-U''$. Then $\overline{U}$ is the solution of the following  problem
\begin{equation}\label{linearproblem--difference}
\begin{cases}
(\overline{U}_i)_t+\tilde{\kappa}_i(t,x)(\overline{U}_i)_x=\overline{\phi}_i(t,x) \qquad \text{for $(t,x)\in R(\delta)$ and  $i=1,2$},\\[1mm]
\overline{U}_1=0  \qquad \text{on $x=t$},\\[1mm]
\overline{U}_2=0 \qquad \text{on $x=0$},
\end{cases}
\end{equation}
where
$\overline{\phi}_i(t,x)=\phi_i(t,x,\psi')-\phi_i(t,x,\psi'')-(\kappa_i(t,x,\psi')-\kappa_i(t,x,\psi''))U''_x$ for $i=1,2$.

Similarly to the derivation of \eqref{solutionform-10}, we obtain \eqref{solutionform-16}.
\end{proof}

\medskip
Next, let us introduce a particular subset of $\Sigma(\delta | M_1)$:
\begin{equation}\label{solutionform-17}
\Sigma(\delta | M_1,M_2(\eta))=\big\{\psi(t,x)\,:\, \psi \in \Sigma(\delta | M_1), \Lambda(\eta|\chi)\leq M_2(\eta)\big\},
\end{equation}
where $\Lambda(\eta|\chi)$ is the modified modulus of continuity of $\chi$, and $M_2(\eta)$ for $0\leq \eta<\infty$
is a nonnegative value (to be specified later) with the property that
\begin{equation}\label{solutionform-18}
  M_2(\eta)   \rightarrow 0 \qquad \text{as $\eta \rightarrow 0$}.
\end{equation}
Then we have
\begin{lemma}\label{contraction-2}
There exist $M_2(\eta)$ and a positive constant $\delta_5\leq \delta_4$ such that,
for any $0<\delta\leq \delta_5$, $\Phi$  maps $\Sigma(\delta | M_1,M_2(\eta))$ into itself.
\end{lemma}

\begin{proof}
First, it follows from \eqref{curve--1} that
\begin{equation*}
\displaystyle
(X_i)_x(\tau;t,x)= e^{\int_t^\tau(\tilde{\kappa}_i)_y (\nu,X_i(\nu;t,x))\,\text{d}\nu},\, \, \,
(X_i)_t(\tau;t,x)=-\tilde{\kappa}_i(t,x)(X_i)_x(\tau;t,x),
\end{equation*}
which, along with \eqref{tezhengqulv2}, yields that,  for any constant  $\varsigma\in [0, 1]$,
\begin{equation}\label{solutionform-24}
\displaystyle
(\tau_1)_t+\varsigma (\tau_1)_x=\frac{\varsigma-\tilde{\kappa}_1(t,x)}{1-\tilde{\kappa}_1(\tau_1,y_1)}
e^{\int_\tau^t(\tilde{\kappa}_1)_y (\nu,X_i(\nu;t,x))\,\text{d}\nu}.
\end{equation}
Then, for any $\psi\in \Sigma(\delta | M_1)$  on $R(\delta)$,
\begin{equation}\label{solutionform-25}
\displaystyle
\|(\tau_1)_t+\varsigma(\tau_1)_x\|_\delta\leq 1+E_1\iota(\delta).
\end{equation}
Similarly, we have
\begin{equation}\label{solutionform-26}
\displaystyle
\|(\tau_2)_t+\varsigma (\tau_2)_x\|_\delta\leq 1+E_1\iota(\delta),
\end{equation}
which, together with \eqref{solutionform-25}, yields that
\begin{equation}\label{solutionform-26555}
\displaystyle
 \Lambda(\eta|\tau_i)\leq \sup_{0\leq \varsigma \leq 1}\|(\tau_i)_t+\varsigma (\tau_i)_x\|_\delta \eta \leq (1+E_1\iota(\delta))\eta
\qquad \mbox{for $i=1,2$}.
\end{equation}

Second, it  follows from \eqref{solutionform-8}, Lemma  \ref{fujiaxingzhi}, \eqref{solutionform-26555},
and the Gronwall inequality  that
\begin{equation}\label{solutionform-22}
 \Lambda(\eta|\chi^1)\leq \big(1+E_1\sigma(\delta)\big)w((1+E_1\sigma(\delta))\eta |G_t)+E_1\Lambda_0(\eta),
\end{equation}
where
$\Lambda_0(\eta):=\big(\|G_t\|_\delta+1\big)\big(\sigma(\eta)+\sigma(\delta)w(\eta|\Gamma(\psi))\big).$

Denote
$$
\Gamma^*(\psi)=\big\{\kappa_i, (\kappa_i)_x, (\kappa_i)_{U_j}, \phi_i, (\phi_i)_x, (\phi_i)_{U_j}, G_i, (G_i)_t \,:\, i,j=1,2 \big\}.
$$
It follows from   Lemma  \ref{fujiaxingzhi}  that, for any $\psi\in \Sigma(\delta | M_1)$ on $R(\delta)$ for $\delta\leq \delta_4$,
\begin{equation}\label{solutionform-30}
w(\eta|\Gamma(\psi))\leq  E_{1}\big(w^*(\eta)+\Lambda(\eta|\chi)\big),
\end{equation}
where $w^*(\eta)$ denotes the modulus of the continuity of $\Gamma^*(\psi)$ on the domain of
$$
\big\{(t,x,\psi): (t,x)\in R(\delta_4), \ |\psi|\leq 1\big\}.
$$
This implies that
\begin{equation}\label{solutionform-31}
\sigma(\eta)\leq E_{1}w^*(\eta).
\end{equation}
Then there exists a positive constant  $\delta_5\leq \delta_4$ such that, for any $0<\delta\leq \delta_5$,
\begin{equation}\label{solutionform-32}
 \Lambda(\eta|\chi^1)\leq  \zeta_1\Lambda(\eta|\chi)+E_{1}\big(w^*(\eta)+\sigma(\eta)\big)\qquad\,\,\mbox{on $R(\delta)$},
\end{equation}
where we have used estimate \eqref{solutionform-22},  and constant $\zeta_1$  satisfies
$0<\zeta_1<1$.

Set
$$
M_2(\eta)=\frac{E_{1}}{1-\zeta_1}\big(w^*(\eta)+\sigma(\eta)\big).
$$
Then  it follows from \eqref{solutionform-32} that, if  $ \Lambda(\eta|\chi)\leq M_2(\eta)$,
$$
\Lambda(\eta|\chi^1)\leq M_2(\eta) \qquad\mbox{on $R(\delta)$ for $\delta\leq \delta_5$}.
$$
This completes the proof of this lemma.
\end{proof}

\smallskip
\noindent
\textbf{4}. {\it Local existence of the nonlinear problem}.
We now prove this theorem. It follows from Lemmas \ref{contraction}--\ref{contraction-2} that,
for any $0<\delta\leq \delta_5$,  $\Phi$  is a contraction operator (with respect to the $C^0$--norm)
mapping  $\Sigma(\delta | M_1,M_2(\eta))$ into itself.
Moreover, $\Sigma(\delta | M_1,M_2(\eta))$  is clearly a compact, closed set in $C^0(R(\delta))$.
Therefore, for $0<\delta\leq \delta_5$, $\Phi$ possesses a unique fixed point $U\in \Sigma(\delta | M_1,M_2(\eta))$,
{\it i.e.}, $\Phi U=U$, which  admits a $C^1$ solution on $R(\delta)$
to the boundary value  problem \eqref{typicalboundaryproblem}.
\end{proof}

\newpage

\bibliographystyle{siamplain}

\begin{thebibliography}{99}

\bibitem{G3}
{\sc G. Chen},
{\em Formation of singularity and smooth wave propagation for the
Nonisentropic compressible Euler
equations}, J. Hyperbolic Differ. Equ., 8 (2011), pp.  671--690.

\bibitem{G9}
{\sc G. Chen},
{\em Optimal time-dependent lower bound on density for classical
 solutions of 1-D compressible Euler equations},
Indiana Univ. Math. J., 66 (2017), pp. 725--740.

\bibitem{CPZ}
{\sc G. Chen, R. Pan, and S. Zhu},
{\em Singularity formation for the compressible Euler equations},
{SIAM J. Math. Anal.}, 49 (2017), pp. 2591--2614.

\bibitem{CPZ2}
{\sc G. Chen, R. Pan, and S. Zhu},
{\em A polygonal scheme and the lower bound on density for the isentropic gas dynamics},
{Discrete Contin. Dyn. Syst. (A)}, 39 (2019), pp. 4259--4277.

\bibitem{G5}
{\sc G. Chen and R. Young},
{\em Smooth solutions and singularity formation for the inhomogeneous nonlinear wave equation},
J. Differential Equations, 252 (2012), pp. 2580--2595.

\bibitem{G6}
{\sc  G. Chen and R. Young},
  {\em Shock-free solutions of the compressible Euler equations},
  Arch. Ration. Mech. Anal., 217 (2015), pp. 1265--1293.

\bibitem{G8}
{\sc G. Chen, R. Young, and  Q. Zhang},
{\em Shock formation in the compressible Euler equations and related systems},
{J. Hyperbolic Differ. Equ.}, 10  (2013), pp. 149--172.

\bibitem{Chen1992}
{\sc G.-Q.  Chen},
{\em The method of quasidecoupling for discontinuous solutions to conservation laws},
  Arch. Rational Mech. Anal., 121 (1992), pp. 131--185.

\bibitem{Chenshuxing}
{\sc S. Chen and L. Dong},
{\em Formation and construction of shock for p-system},
{Science in China}, 44  (2001),  pp. 1139--1147.

\bibitem{Chenshuxing2}
{\sc S. Chen,  Z. Xin, and H. Yin},
{\em Formation and construction of shock wave for quasilinear hyperbolic system and its application
to inviscid compressible flow},
Research Report, IMS, CUHK (1999).

\bibitem{Chris}
{\sc D. Christodoulou, }
{\em The Shock Development Problem},
EMS Monographs in Mathematics, EMS Publishing House, 2019.

\bibitem{shuang}
{\sc D. Christodoulou and S. Miao},
{\em Compressible Flow and Euler's Equations},
International Press: Somerville, MA; Higher Education Press, Beijing, 2014.

\bibitem{courant}
{\sc R. Courant  and  K. O. Friedrichs},
  {\em Supersonic Flow and Shock Waves},
Interscience Publishers, Inc.: New York, 1948.

\bibitem{Dafermos2010}
{\sc C.~M. Dafermos},
{\em Hyperbolic Conservation Laws in Continuum Physics},
 4th Edition, Springer-Verlag: Berlin, 2016.

\bibitem{Jenssen}
{\sc H. K. Jenssen},
{\em  On exact solutions of rarefaction-rarefaction interactions in
   compressible isentropic flow},
J. Math. Fluid Mech., 19 (2017), pp. 685--708.

\bibitem{Fritzjohn}
{\sc F. John},
{\em Formation of singularities in one-dimensional nonlinear wave propagation},
{Comm. Pure Appl. Math.}, 27 (1974), pp. 377--405.


\bibitem{Kongdexing}
{\sc D. Kong},
{\em Formation and propagation of singularities for $2\times2$ quasilinear hyperbolic systems},
{Trans. Amer. Math. Soc.}, 354 (2002), pp. 3155--3179.


\bibitem{lax2}{\sc P. Lax},
{\em Development of singularities of solutions of nonlinear hyperbolic partial differential equations},
{J. Math. Phys.}, 5 (1964), pp. 611--613.


\bibitem{lebaud}
{\sc M. Lebaud},
{\em Description de la formation $d^{'}$un choc dans le p-systeme},
{J. Math Pures Appl.}, 73 (1994), pp. 523--565.	

\bibitem{Lidaqian}
{\sc T. Li},
{\em Global Classical Solutions for Quasilinear Hyperbolic Systems},
Research in Applied Math. 32, Wiley-Masson, 1994.

\bibitem{liyu}
{\sc T. Li and W. Yu},
{\em Boundary Value Problems for Quasilinear Hyperbolic Systems},
Duke University Press: Durham, 1985.

\bibitem{Yachun}
{\sc Y. Li and  S. Zhu},
{\em On regular solutions of the $3$-D compressible isentropic Euler-Boltzmann equations with vacuum},
{Discrete Contin. Dynam. Systems-A}, 35 (2015), pp. 3059--3086.

\bibitem{lin2}
{\sc L. Lin},
{\em On the vacuum state for the equations of isentropic gas dynamics},
{J. Math. Anal. Appl.},
121 (1987), pp. 406--425.


\bibitem{linliuyang}
 {\sc L. Lin, H. Liu, and T. Yang},
{\em Existence of globally bounded continuous solutions for
  nonisentropic gas dynamics equations},
 J. Math. Anal. Appl., 209 (1997), pp. 492--506.

\bibitem{liu}
{\sc T. Liu},
{\em Development of singularities in the nonlinear waves for quasi-linear hyperbolic partial differential equations},
{J. Differential Equations}, 33 (1979), pp. 92--111.


\bibitem{ls}
{\sc T. Liu and J. Smoller},
\emph{On the vacuum state for the isentropic gas dynamics equations},
Adv. in Appl. Math., 1 (1980), pp. 345--359.

\bibitem{Speck}
{\sc J. Luk and  J. Speck},
{\em Shock formation in solutions to the 2D compressible Euler equations in the presence of non-zero vorticity},
Invent. Math., 214 (2018), pp. 1--169.

\bibitem{Makino}
{\sc T. Makino,  S. Ukai, and S. Kawashima},
{\em Sur la solution $\grave{\text{a}}$ support compact de equations d'Euler compressible},
{Japan J. Appl. Math.},  33 (1986), pp. 249--257.

\bibitem{huahua}
{\sc R. Pan and Y. Zhu},
{\em Singularity formation for one dimensional  full Euler equations},
 J. Differential Equations,  261 (2016),  pp. 7132--7144.

\bibitem{Rammaha}
{\sc M. A. Rammaha},
{\em Formation of singularities in compressible fluids in two space dimensions},
{Proc. Amer. Math. Soc.}, 107 (1989), pp. 705--714.

\bibitem{Riemann}
{\sc B. Riemann},
\emph{Ueber die Fortpflanzung ebener Luftwellen von endlicher Schwingungsweite}, Abhandlungen
der Kniglichen Gesellschaft der Wissenschaften zu Gttingen, 8 (1860), 43.

\bibitem{Sideris}
{\sc T. C. Sideris},
{\em Formation of singularities in three-dimensional compressible fluids},
{Commun. Math. Phys.}, 101 (1985), pp. 475--487.

\bibitem{smoller}
{\sc J. Smoller},
{\em Shock Waves and Reaction-Diffusion Equations},
Springer-Verlag: New York-Berlin, 1983.

\bibitem{Stokes}{\sc G. G. Stokes},
{\em On a difficulty in the theory of sound},
Philos. Mag. Ser. 3, 33 (1848),  pp. 349--356.

\bibitem{TW}
{\sc E. Tadmor and  D. Wei},
{\em On the global regularity of subcritical Euler-Poisson equations
   with pressure},
J. Eur. Math. Soc., 10 (2008),  pp. 757--769.

\bibitem{youngblake1}
{\sc B. Temple and R. Young},
\emph{A paradigm for time-periodic sound wave propagation in the
compressible Euler equations,} {Methods Appl. Anal.},
16 (2009), pp. 341--364.

\bibitem{Wagner1987}
{\sc D. Wagner},
{\em Equivalence of the Euler and Lagrangian equations of gas dynamics for weak solutions},
J. Differential Equations, 68 (1987), pp. 118--136.

\bibitem{Whitney} {\sc H. Whitney},
{\em On singularities of mappings of Euclidean space}, I{\rm :} {\em Mappings of the plane into the plane},
Ann. of Math., 62 (1955), pp. 374--410.
\end{thebibliography}

\end{document}